\documentclass[%
  a4paper,
  onecolumn,
  colorlinks,
]{preprint}

\usepackage[noadjust]{cite}

\usepackage[english]{babel}

\usepackage{amsmath}
\usepackage{amssymb}
\usepackage{amsthm}
\usepackage{mathtools}
\usepackage{dsfont}

\usepackage[linesnumbered, lined, algoruled]{algorithm2e}

\usepackage{xcolor}
\usepackage{graphicx}

\usepackage{tikz}
\usepackage{pgfplots}
  \pgfplotsset{%
    compat        = 1.17,
    table/col sep = comma}
  \usepgfplotslibrary{external}
  \tikzset{external/system call = {%
    pdflatex \tikzexternalcheckshellescape
      -halt-on-error
      -interaction=batchmode
      -jobname "\image" "\texsource"}}
  \tikzexternaldisable

\usepackage{caption}
\usepackage{subcaption}

\newcommand{\C}{\ensuremath{\mathbb{C}}}
\newcommand{\R}{\ensuremath{\mathbb{R}}}
\newcommand{\N}{\ensuremath{\mathbb{N}}}

\newcommand{\M}{\ensuremath{\mathcal{M}}}
\newcommand{\Mk}{\ensuremath{\mathcal{M}^{(k)}}}
\newcommand{\MConj}{\ensuremath{\overline{\M}}}
\newcommand{\MAll}{\ensuremath{\widetilde{\M}}}
\newcommand{\Pset}{\ensuremath{\mathcal{P}}}
\newcommand{\PConj}{\ensuremath{\overline{\Pset}}}
\newcommand{\PAll}{\ensuremath{\widetilde{\Pset}}}
\newcommand{\Pk}{\ensuremath{\Pset^{(k)}}}
\newcommand{\PSO}{\ensuremath{\Pset_{\operatorname{SO}}}}
\newcommand{\PSOConj}{\ensuremath{\PConj_{\operatorname{SO}}}}
\newcommand{\PSOAll}{\ensuremath{\PAll_{\operatorname{SO}}}}
\newcommand{\PkSO}{\ensuremath{\Pk_{\operatorname{SO}}}}

\newcommand{\trans}{\ensuremath{\mkern-1.5mu\mathsf{T}}}
\DeclareMathOperator{\mdiag}{diag}
\DeclareMathOperator*{\argmax}{argmax}
\DeclareMathOperator*{\argmin}{argmin}
\DeclareMathOperator{\real}{Re}
\DeclareMathOperator{\imag}{Im}

\newcommand{\mA}{\ensuremath{\boldsymbol{A}}}
\newcommand{\mD}{\ensuremath{\boldsymbol{D}}}
\newcommand{\mE}{\ensuremath{\boldsymbol{E}}}
\newcommand{\mG}{\ensuremath{\boldsymbol{G}}}
\newcommand{\mJ}{\ensuremath{\boldsymbol{J}}}
\newcommand{\mK}{\ensuremath{\boldsymbol{K}}}
\newcommand{\mM}{\ensuremath{\boldsymbol{M}}}
\newcommand{\mV}{\ensuremath{\boldsymbol{V}}}

\newcommand{\malpha}{\ensuremath{\boldsymbol{\alpha}}}
\newcommand{\mbeta}{\ensuremath{\boldsymbol{\beta}}}
\newcommand{\mLambda}{\ensuremath{\boldsymbol{\Lambda}}}
\newcommand{\mSigma}{\ensuremath{\boldsymbol{\Sigma}}}
\newcommand{\mPsi}{\ensuremath{\boldsymbol{\Psi}}}

\newcommand{\ETAk}{\boldsymbol{\eta}^{(k)}}
\newcommand{\etakDenomWeighted}{%
  \ensuremath{\boldsymbol{\eta}^{(k)}_{\hat{\operatorname{d}}}}}
\newcommand{\etaDenomWeighted}{%
  \ensuremath{\boldsymbol{\eta}_{\hat{\operatorname{d}}}}}
\newcommand{\ETAReal}{\ensuremath{\boldsymbol{\eta}_{\operatorname{real}}}}

\newcommand{\gWeighted}{\ensuremath{\mg_{\boldsymbol{\eta}}}}

\newcommand{\mb}{\ensuremath{\boldsymbol{b}}}
\newcommand{\mc}{\ensuremath{\boldsymbol{c}}}
\newcommand{\mf}{\ensuremath{\boldsymbol{f}}}
\newcommand{\mg}{\ensuremath{\boldsymbol{g}}}
\newcommand{\mh}{\ensuremath{\boldsymbol{h}}}
\newcommand{\mr}{\ensuremath{\boldsymbol{r}}}
\newcommand{\mw}{\ensuremath{\boldsymbol{w}}}
\newcommand{\mx}{\ensuremath{\boldsymbol{x}}}
\newcommand{\mz}{\ensuremath{\boldsymbol{z}}}

\newcommand{\wk}{\ensuremath{w^{(k)}}}
\newcommand{\mwk}{\ensuremath{\mw^{(k)}}}
\newcommand{\mwReal}{\ensuremath{\mw_{\operatorname{real}}}}
\newcommand{\mwConj}{\ensuremath{\overline{w}}}

\newcommand{\mgk}{\ensuremath{\mg^{(k)}}}

\newcommand{\mgReal}{\ensuremath{\mg_{\operatorname{real}}}}

\newcommand{\ek}{\ensuremath{\hat{\mr}^{(k)}}}
\newcommand{\ekSO}{\ensuremath{\ek}_{\operatorname{SO}}} 
\newcommand{\ekTrue}{\ensuremath{\mr^{(k)}}}
\newcommand{\ekTrueSO}{\ensuremath{\ekTrue_{\operatorname{SO}}}}
\newcommand{\eSO}{\ensuremath{\hat{\mr}_{\operatorname{SO}}}}
\newcommand{\eTrueSO}{\ensuremath{\mr_{\operatorname{SO}}}}
\newcommand{\ekResults}{\ensuremath{\boldsymbol{\epsilon}(\M)}}

\newcommand{\arbErrVec}{\ensuremath{\mz^{(k)}}}

\newcommand{\LtwoWeighted}{\ensuremath{\epsilon_{\Ltwo, \boldsymbol{\eta}}}}
\newcommand{\LinfWeighted}{\ensuremath{\epsilon_{\Linf, \boldsymbol{\eta}}}}
\newcommand{\relerr}{\ensuremath{\epsilon_{\operatorname{ptw}}}}

\newcommand{\JacSOk}{\ensuremath{\mJ^{(k)}_{\operatorname{SO}}}}
\newcommand{\JacSONLkFull}{\ensuremath{\mJ_{\operatorname{SONL}}}}
\newcommand{\JacSONLkW}{\ensuremath{\JacSONLkFull(\mwk)}}
\newcommand{\JacSONLkSig}{\ensuremath{\JacSONLkFull(\sigkVec)}}
\newcommand{\JacSO}{\ensuremath{\mJ_{\operatorname{SO}}}}
\newcommand{\JacSONoConj}{\ensuremath{\JacSO(\msigma)}}
\newcommand{\JacSOConj}{\ensuremath{\JacSO(\msigmaConj)}}
\newcommand{\JacRealSONLkW}{%
  \ensuremath{\widetilde{\mJ}_{\operatorname{SONL}}(\mw)}}
\newcommand{\JacRealSONLkSig}{%
  \ensuremath{\widetilde{\mJ}_{\operatorname{SONL}}(\msigma)}}
\newcommand{\JacRealSONLkSigConj}{%
  \ensuremath{\widetilde{\mJ}_{\operatorname{SONL}}(\msigmaConj)}}
\newcommand{\JacRealSONLkFull}{%
  \ensuremath{\widetilde{\mJ}_{\operatorname{SONL}}(\mw, \msigma)}}
\newcommand{\JacRealSONLkWConj}{%
  \ensuremath{\widetilde{\mJ}_{\operatorname{SONL}}(\mwConj)}}
\newcommand{\JacRealSONLkFullConj}{%
  \ensuremath{\widetilde{\mJ}_{\operatorname{SONL}}(\mwConj,\msigmaConj)}}

\newcommand{\msigma}{\ensuremath{\boldsymbol{\sigma}}}
\newcommand{\sigk}{\ensuremath{\sigma^{(k)}}}
\newcommand{\sigkVec}{\ensuremath{\msigma}^{(k)}}
\newcommand{\msigmaConj}{\ensuremath{\overline{\msigma}}}

\newcommand{\Hr}{\ensuremath{\widehat{H}}}
\newcommand{\Hrk}{\ensuremath{\widehat{H}^{(k)}}}
\newcommand{\Hso}{\ensuremath{H_{\operatorname{SO}}}}
\newcommand{\HrSO}{\ensuremath{\Hr_{\operatorname{SO}}}}
\newcommand{\HrkSO}{\ensuremath{\Hrk_{\operatorname{SO}}}}

\newcommand{\numerUSk}{\ensuremath{\hat{n}^{(k)}}}
\newcommand{\denomUSk}{\ensuremath{\hat{d}^{(k)}}}

\newcommand{\nso}{\ensuremath{n_{\operatorname{SO}}}}
\newcommand{\dso}{\ensuremath{d_{\operatorname{SO}}}}

\newcommand{\numerSOk}{\ensuremath{\nso^{(k)}}}
\newcommand{\denomSOk}{\ensuremath{\dso^{(k)}}}

\newcommand{\FOM}{\ensuremath{G}}

\newcommand{\LowUS}{\ensuremath{\mathbb{L}}}
\newcommand{\LowUSk}{\ensuremath{\LowUS^{(k)}}}
\newcommand{\LowSO}{\ensuremath{\LowUS_{\operatorname{SO}}}}
\newcommand{\LowSOk}{\ensuremath{\LowUSk}_{\operatorname{SO}}}
\newcommand{\LowSONL}{\ensuremath{\LowUS_{\operatorname{SONL}}}}
\newcommand{\LowSONLk}{\ensuremath{\LowUSk_{\operatorname{SONL}}}}
\newcommand{\LowUStoSO}{\ensuremath{\widehat{\LowUS}_{\operatorname{SO}}}}
\newcommand{\LowUSReal}{\ensuremath{\breve{\LowUS}}}
\newcommand{\LowSOReal}{\ensuremath{\breve{\LowUS}_{\operatorname{SO}}}}

\theoremstyle{plain}\newtheorem{theorem}{Theorem}
\theoremstyle{plain}\newtheorem{lemma}{Lemma}
\theoremstyle{plain}\newtheorem{corollary}{Corollary}
\theoremstyle{definition}\newtheorem{remark}{Remark}

\newcommand{\eye}[1]{\ensuremath{\boldsymbol{I}_{#1}}}
\newcommand{\ones}[1]{\ensuremath{\mathds{1}_{#1}}}
\newcommand{\maxOrder}{\ensuremath{k_{\max}}}
\newcommand{\imunit}{\ensuremath{\mathfrak{i}}}
\newcommand{\Linf}{\ensuremath{\mathcal{L}_{\infty}}}
\newcommand{\Ltwo}{\ensuremath{\mathcal{L}_{2}}}

\newcommand{\VP}{\texttt{VarPro}}
\newcommand{\AAA}{\texttt{AAA}}
\newcommand{\AAAtwo}{\texttt{AAA2}}
\newcommand{\AAAFull}{\texttt{SO-AAA}}
\newcommand{\AAALin}{\texttt{LSO-AAA}}
\newcommand{\AAAFullNL}{\texttt{NSO-AAA}}

\newcommand{\plotfontsize}{\small}

\definecolor{matlabblue}{HTML}{0072BD}
\definecolor{matlaborange}{HTML}{D95319}
\definecolor{matlabyellow}{HTML}{EDB120}
\definecolor{matlabpurple}{HTML}{7E2F8E}
\definecolor{matlabgreen}{HTML}{77AC30}
\definecolor{matlablightblue}{HTML}{4DBEEE}
\definecolor{matlabred}{HTML}{A2142F}

\colorlet{trueDataColor}{black}
\colorlet{AAAColor}{matlabblue}
\colorlet{SOAAAColor}{matlabpurple}
\colorlet{LSOAAAColor}{matlaborange}
\colorlet{NLSOAAAColor}{matlabgreen}
\colorlet{AAA2Color}{matlabred}

\tikzstyle{trueData} = [
  trueDataColor,
  solid,
  only marks,
  mark size = 1.75pt,
  mark      = *
]

\tikzstyle{AAAConverge} = [
  AAAColor,
  dashed,
  line width   = 1.25pt,
  mark         = x,
  mark size    = 1.75pt,
  mark options = {solid},
  mark repeat  = {2}
]

\tikzstyle{SOAAAConverge} = [
  SOAAAColor,
  solid,
  line width   = 1.252pt,
  mark         = *,
  mark size    = 1.75pt,
  mark options = {solid},
  mark repeat  = {2}
]

\tikzstyle{LSOAAAConverge} = [
  LSOAAAColor,
  dotted,
  line width   = 1.25pt,
  mark         = diamond*,
  mark size    = 1.75pt,
  mark options = {solid},
  mark repeat  = {2}
]

\tikzstyle{NLSOAAAConverge} = [
  NLSOAAAColor,
  dashdotted,
  line width   = 1.25pt,
  mark         = square*,
  mark size    = 1.5pt,
  mark options = {solid},
  mark repeat  = {2}
]

\tikzstyle{AAA2kConverge} = [
  AAA2Color,
  dashed,
  line width   = 1.25pt,
  mark         = triangle*,
  mark size    = 1.75pt,
  mark options = {solid},
  mark repeat  = {2}
]

\tikzstyle{AAAResponse} = [
  AAAConverge,
  mark repeat = {50},
  mark phase  = {0}
]

\tikzstyle{SOAAAResponse} = [
  SOAAAConverge,
  mark repeat = {50},
  mark phase  = {10}
]

\tikzstyle{LSOAAAResponse} = [
  LSOAAAConverge,
  mark repeat = {50},
  mark phase  = {20}
]

\tikzstyle{NLSOAAAResponse} = [
  NLSOAAAConverge,
  mark repeat = {50},
  mark phase  = {30}
]

\tikzstyle{AAA2kResponse} = [
  AAA2kConverge,
  mark repeat = {50},
  mark phase  = {40}
]

\begin{document}


\title{Second-order AAA algorithms for structured data-driven modeling}

\author[$\ast$]{Michael S. Ackermann}
\affil[$\ast$]{Department of Mathematics, Virginia Tech, Blacksburg,
  VA 24061, USA.\authorcr
  \email{amike98@vt.edu}, \orcid{0000-0003-3581-6299}}

\author[$\dagger$]{Ion Victor Gosea}
\affil[$\dagger$]{
  Max Planck Institute for Dynamics of Complex Technical Systems,
  Sandtorstr. 1, 39106 Magdeburg, Germany.\authorcr
  \email{gosea@mpi-magdeburg.mpg.de}, \orcid{0000-0003-3580-4116}
}
  
\author[$\ddagger$]{Serkan Gugercin}
\affil[$\ddagger$]{%
  Department of Mathematics and Division of Computational Modeling and Data
  Analytics, Academy of Data Science, Virginia Tech,
  Blacksburg, VA 24061, USA.\authorcr
  \email{gugercin@vt.edu}, \orcid{0000-0003-4564-5999}
}

\author[$\S$]{Steffen W. R. Werner}
\affil[$\S$]{%
  Department of Mathematics, Division of Computational Modeling and
  Data Analytics, and Virginia Tech National Security Institute, Virginia Tech,
  Blacksburg, VA 24061, USA.\authorcr
  \email{steffen.werner@vt.edu}, \orcid{0000-0003-1667-4862}
}

\shorttitle{Second-order AAA algorithms}
\shortauthor{Ackermann, Gosea, Gugercin, Werner}
\shortdate{2025-06-02}
\shortinstitute{}



\abstract{%
  The data-driven modeling of dynamical systems has become an essential
  tool for the construction of accurate computational models from real-world
  data.
  In this process, the inherent differential structures underlying the
  considered physical phenomena are often neglected making the reinterpretation
  of the learned models in a physically meaningful sense very challenging.
  In this work, we present three data-driven modeling approaches for
  the construction of dynamical systems with second-order differential
  structure directly from frequency domain data.
  Based on the second-order structured barycentric form, we extend the
  well-known Adaptive Antoulas-Anderson algorithm to the case of second-order
  systems.
  Depending on the available computational resources, we propose variations of
  the proposed method that prioritize either higher computation speed or
  greater modeling accuracy, and
  we present a theoretical analysis for the expected accuracy and performance
  of the proposed methods.
  Three numerical examples demonstrate the effectiveness of our new structured
  approaches in comparison to classical unstructured data-driven modeling.
}

\novelty{%
  We develop three new computational methods for the data-driven modeling of
  second-order dynamical systems from frequency domain data.
  The proposed methods adaptively determine the modeling order needed to satisfy
  a given error tolerance.
  We provide a theoretical analysis of the expected accuracy and performance
  for all proposed methods.
}

\maketitle


\section{Introduction}%
\label{sec:intro}

The data-driven modeling of dynamical systems is an essential tool for the
construction of high-fidelity models of physical phenomena when first-principle
formulations are not available yet abundant input/output data are.
Of particular importance for the meaningful physical interpretation of
data-based models is the integration of differential structures into the
modeling process.
In this work, we consider linear dynamical systems that are described by
second-order ordinary differential equations of the form
\begin{subequations} \label{eqn:MechSystem}
\begin{align}
  \mM \ddot{\mx}(t) + \mD \dot{\mx}(t) + \mK \mx(t) & = \mb u(t),\\
  y(t) & = \mc^{\trans} \mx(t),
\end{align}
\end{subequations}
where $\mM \in \C^{k \times k}$ is nonsingular,
$\mD, \mK \in \C^{k \times k}$,
$\mb \in \C^{k}$ and $\mc \in \C^{k}$.
The quantities $\mx(t)$, $u(t)$ and $y(t)$ in~\cref{eqn:MechSystem} denote the
internal system state, the external input and output of the system,
respectively.
Systems like~\cref{eqn:MechSystem} typically appear in the modeling process of
mechanical, electrical and acoustical
structures~\cite{AbrM87, AumW23, Lob18, Bla18a}.
The input-to-output behavior of the system~\cref{eqn:MechSystem} can
be equivalently described in the frequency (Laplace) domain by the
corresponding transfer function
\begin{equation} \label{eqn:MechTransfer}
  \Hso(s) = \mc^{\trans} (s^{2} \mM + s \mD + \mK)^{-1} \mb,
\end{equation}
with the complex variable $s \in \C$. 

In the setting of data-driven modeling that we are interested in, we do not
assume access to internal state-space data, as in~\eqref{eqn:MechSystem}, but
only the input-output data in the form of frequency domain data.
In other words, we have only access to transfer function measurements~$g_{i}$
of an underlying dynamical system at the sampling (frequency) points~$\mu_{i}$,
for $i = 1, \ldots, N$.
Then, the task we consider in this work is the construction of structured
dynamical systems of the form~\cref{eqn:MechSystem} from given transfer function
measurements $\{ (\mu_{i}, g_{i}) \}_{i = 1}^{N}$, so that the data is
approximated well by the corresponding transfer function
\begin{equation}
  \lVert g_{i} - \Hso(\mu_{i}) \rVert \leq \tau
  \quad\text{for}\quad
  i = 1, \ldots, N,
\end{equation}
in some suitable norm and with a user-defined tolerance $\tau \geq 0$.
This is a challenging problem because (i) the fitting of system parameters
towards the given data typically leads to nonlinear, nonconvex optimization
problems, and (ii) the recovery of the second-order structure from an
unstructured transfer function is usually not possible.
In this work, we will utilize the structured barycentric forms developed
in~\cite{GosGW24} to derive a series of new structure-enforcing greedy
data-driven modeling methods to construct second-order
systems~\cref{eqn:MechSystem} directly from frequency domain data.

In the last decades, there have been many significant advances in the area of
data-driven modeling of dynamical systems from frequency domain data for
the case of unstructured (first-order) systems.
This includes transfer function interpolation via the
Loewner framework~\cite{AntA86, MayA07},
nonlinear least-squares approaches like vector fitting~\cite{GusS99, DrmGB15}
and the RKFIT method~\cite{BerG15, BerG17}, and greedy approaches
like the Adaptive Antoulas-Anderson (AAA) algorithm~\cite{NakST18}.
In particular, the AAA algorithm has gained significantly in popularity for
the modeling of dynamical systems in the frequency domain and
has been successfully applied in different applications ranging from
power systems~\cite{MonJIetal20}
to optics~\cite{BetHZetal24}
to acoustics~\cite{BraGAetal25}, and it has been extended to fitting
multivariate problems, e.g., transfer functions of parametric dynamical
systems~\cite{RodBG23}.
In recent years, there have been efforts to extend existing data-driven
modeling approaches to second-order systems~\cref{eqn:MechSystem},
like the interpolatory Loewner framework in~\cite{PonGB22, SchUBetal18} or
the vector fitting method for least-squares fitting in~\cite{WerGG22}.
In contrast to these existing works, we will provide extensions for the greedy
AAA algorithm towards the data-driven modeling of second-order
systems~\cref{eqn:MechSystem}.

The remainder of this manuscript is organized as follows.
In \Cref{sec:Preliminaries}, we outline the mathematical foundations for this
paper by establishing the connection between dynamical systems and rational
functions, introducing barycentric forms, and reminding the reader of the basics
of the varibale projection method for solving nonlinear least-squares problems.
Then, we introduce the unstructured AAA algorithm for the type of the barycentric form
that we consider in this paper in \Cref{sec:AAA_Unstructured}.
From there, we extend the AAA algorithm to the setting of structured
second-order systems in \Cref{sec:SecondOrderAAAs} to propose three new
structured variants of the method.
In \Cref{sec:ModelStructureAndAccuracyHeuristics}, we analyze the potential
performance of our proposed methods in comparison to the classical approach
and in \Cref{sec:realification}, we introduce modifications needed to construct
dynamical systems with real matrices.
The performance of the proposed methods is tested in \Cref{sec:numerics} in
three numerical examples, including the vibrational response of a beam and a
gyroscope as well as the acoustical behavior of a damped cavity.
The paper is concluded in \Cref{sec:conclusions}.


\section{Mathematical preliminaries}%
\label{sec:Preliminaries}

In this section, we recap the foundations of unstructured and structured
barycentric forms as used in the data-driven modeling of dynamical systems.
We will also remind the reader of the variable projection method for nonlinear
least-squares problems as we will employ it later in some of our proposed
methods.


\subsection{First-order systems and the unstructured barycentric form}%
\label{sec:UnstructSystems}

A linear dynamical system of the form
\begin{subequations} \label{eqn:UnstructuredSystem}
\begin{align}
  \mE \dot{\mx}(t) & = \mA \mx(t) + \mb u(t)\\
  y(t) &= \mc^{\trans} \mx(t),
\end{align}
\end{subequations}
where $\mE \in \C^{k \times k}$ is nonsingular, $\mA \in \C^{k\times k}$,
$\mb \in \C^{k}$, and $\mc \in \C^{k}$, is called a \emph{first-order}
(unstructured) system.
In~\cref{eqn:UnstructuredSystem}, the quantities $\mx(t) \in \C^{k}$,
$u(t) \in \C$, $y(t) \in \C$ are the internal system state, external input
and output, respectively, at time $t \in \R$.
The \emph{state-space dimension} of the system~\cref{eqn:UnstructuredSystem} is
given by $k$, the size of the state vector.
Using the Laplace transform, the system~\cref{eqn:UnstructuredSystem} can
equivalently described in the frequency domain via its transfer function
\begin{equation} \label{eqn:UnstructuredTransfer}
    H(s) = \mc^{\trans} (s\mE - \mA)^{-1} \mb,
\end{equation}
with $s \in \C$.
The transfer function in~\cref{eqn:UnstructuredTransfer} is a strictly
proper degree-$k$ rational function in the complex variable $s$.
Note that any strictly proper rational function of degree $k$ can be written in
the form~\cref{eqn:UnstructuredTransfer}.

A powerful tool for numerical computation with rational functions is the
\emph{barycentric form}, given by 
\begin{equation} \label{eqn:UnstructuredBaryForm}
  H(s) = \frac{n(s)}{d(s)} =
    \frac{\displaystyle \sum\limits_{j = 1}^{k} \frac{h_{j} w_{j}}%
    {s - \lambda_{j}}}%
    {\displaystyle 1 + \sum\limits_{j = 1}^{k} \frac{w_{j}}{s - \lambda_{j}}},
\end{equation}
where $\lambda_{j}, w_{j}, h_{j} \in \C$ and $w_{j} \neq 0$,
for $j = 1, \ldots, k$.
The $\lambda_{j}$'s are referred to as the \emph{barycentric support points},
the $w_{j}$'s are the called \emph{barycentric weights}, and
the $h_{j}$'s are \emph{function values}.
Provided that all $w_{j} \neq 0$, the barycentric
model~\cref{eqn:UnstructuredBaryForm} has a removable singularity at each
support point $\lambda_{j}$ and satisfies $\lim_{s \to \lambda_{j}} = h_{j}$.
Therefore, the particular form~\cref{eqn:UnstructuredBaryForm} is also known as
the \emph{interpolatory barycentric form}.
In the remainder of this work, we will consider the analytic continuation of
the barycentric form~\cref{eqn:UnstructuredBaryForm} to include its support
points, that is, we will write $H(\lambda_{j}) = h_{j}$.
Furthermore, we assume that rational functions are irreducible.
That means the rational functions have no pole-zero cancellations, i.e.,
numerator and denominator do not have common roots.
As the following result shows, by varying the weights $w_{j}$ in the barycentric
form~\cref{eqn:UnstructuredBaryForm}, any strictly proper irreducible
degree-$k$ rational function that takes the value $h_{j}$ at $\lambda_{j}$,
for~$j = 1, \ldots, k$, can be recovered.

\begin{lemma} \label{lmm:BaryFormRecoverAnything}
  Any irreducible strictly proper rational function $f$ of degree $k$ can be
  represented by~\cref{eqn:UnstructuredBaryForm}.
  Furthermore, the distinct support points $\lambda_{j}$,
  for $j = 1, \ldots, k$, can be chosen arbitrary, provided $f$ has no pole at
  $\lambda_{j}$.
\end{lemma}
\begin{proof}
  Let $f$ be a strictly proper rational function of degree-$k$.
  Then, the function $f$ has $2k$ degrees of freedom, and any strictly proper
  rational function $\Hr$ of degree $k$, which interpolates $f$ at
  $2k$ distinct points, must satisfy $\Hr \equiv f$ on $\C$.
  Let
  \begin{equation} \label{eqn:tmp_interpolationData_inproof}
    \{ (\lambda_{j}, h_{j})\}_{j = 1}^{k} \quad\text{and}\quad
    \{ (\mu_{i}, g_{i}) \}_{i = 1}^{k},
  \end{equation}
  where $h_{j} = f(\lambda_{j})$ and $g_{i} = f(\mu_{i})$, be two data sets
  so that $\lambda_{j}$ and $\mu_{i}$ are distinct and no poles of $f$,
  for $j, i = 1, \ldots, k$.
  By construction, the barycentric form~\cref{eqn:UnstructuredBaryForm}
  satisfies $\Hr(\lambda_{j}) = h_{j}$, for all $j = 1, \ldots, k$.
  Then, it is left to show that there exist barycentric weights
  \begin{equation} \label{eqn:tmp_Baryweights_inproof}
    \mw = \begin{bmatrix} w_{1} & \ldots & w_{k} \end{bmatrix}^{\trans}
  \end{equation}
  such that~\cref{eqn:UnstructuredBaryForm} satisfies $\Hr(\mu_{i}) = g_{i}$,
  for all $i = 1, \ldots, k$.
  Thus, for fixed support points and function values, the barycentric weights
  need to be chosen so that
  \begin{equation} \label{eqn:tmp_InterConds_inproof}
    \frac{\displaystyle \sum_{j = 1}^{k} \frac{h_{j} w_{j}}%
    {\mu_{i} - \lambda_{j}}}%
    {\displaystyle 1 + \sum_{j = 1}^{k} \frac{w_{j}}{\mu_{i} - \lambda_{j}}}
    = g_{i},
  \end{equation}
  holds for $i = 1, \ldots, k$.
  Define the Loewner matrix $\LowUS \in \C^{k \times k}$ and the
  vector of function values $\mg \in \C^{k}$ via
  \begin{equation}
    \LowUS_{i, j} = \frac{g_{i} - h_{j}}{\mu_{i} - \lambda_{j}}
    \quad \text{and}\quad
    \mg_{i} = g_{i}, \quad\text{for}~i, j = 1, \ldots, k.
  \end{equation}
  The constraints~\cref{eqn:tmp_InterConds_inproof} are equivalent to the
  existence of $\mw \in \C^{k}$ such that
  \begin{equation} \label{eqn:tmp_FitExtraInterp_inproof}
    -\LowUS \mw = \mg.
  \end{equation}
  Since the data~\cref{eqn:tmp_interpolationData_inproof} comes from a rational
  function of degree $k$, the Loewner matrix $\LowUS$ has full
  rank~\cite{AntA86}.
  Therefore, there exists the vector $\mw$ as the solution
  to~\cref{eqn:tmp_FitExtraInterp_inproof}.
  Then, the barycentric form~\cref{eqn:UnstructuredBaryForm}, with the
  parameters $\lambda_{j}, h_{j}$ from~\cref{eqn:tmp_interpolationData_inproof}
  and weights $w_{j}$ that satisfy~\cref{eqn:tmp_FitExtraInterp_inproof},
  interpolates the rational function $f$ at the $2k$
  points~\cref{eqn:tmp_interpolationData_inproof}, which shows the desired
  result that $\Hr \equiv f$.
\end{proof}

The result from \Cref{lmm:BaryFormRecoverAnything} implies that the support
points in~\cref{eqn:UnstructuredBaryForm} may be chosen as any set of $k$
distinct values in $\C$, which are not the locations of poles of the rational
function.
To relate the barycentric form~\cref{eqn:UnstructuredBaryForm} to the
state-space realization in~\cref{eqn:UnstructuredSystem} and its transfer
function as in~\cref{eqn:UnstructuredTransfer}, define
\begin{subequations} \label{eqn:Lam_w_h_def}
\begin{align}
    \mLambda & = \mdiag(\lambda_{1}, \ldots, \lambda_{k}) \in \C^{k\times k}, \\
    \mw & = \begin{bmatrix} w_{1} & \ldots & w_{k} \end{bmatrix}^{\trans}
      \in \C^{k}, \\
    \mh & = \begin{bmatrix} h_{1} & \ldots & h_{k} \end{bmatrix}^{\trans}
      \in \C^{k}.
\end{align}
\end{subequations}
Also, let $\ones{k} \in \R^{k}$ denote the vector of all ones of length $k$.
Then, a state-space realization of the dynamical
system~\cref{eqn:UnstructuredSystem} (and transfer
function~\cref{eqn:UnstructuredTransfer}) corresponding to the barycentric
form~\cref{eqn:UnstructuredBaryForm} is given via
\begin{equation} \label{eqn:EAbc_def}
  \mE = \eye{k}, \quad
  \mA = \mLambda - \mb \ones{k}^{\trans}, \quad
  \mb = \mw \quad
  \text{and} \quad
  \mc = \mh,
\end{equation}
where $\eye{k} \in \R^{k \times k}$ is the $k$-dimensional identity matrix;
see,~\cite{GosGW24} for further details.


\subsection{Second-order systems and second-order barycentric forms}%
\label{sec:MechBaryForm}

Recently in~\cite{GosGW24}, the barycentric form has been extended to
second-order systems~\cref{eqn:MechSystem} with the structured transfer
function~\cref{eqn:MechTransfer}.
First, we note that~\cref{eqn:MechTransfer} is a degree-$2k$ rational function
in the complex variable $s$.
Also, note that the state-space dimension of~\cref{eqn:MechSystem} is $k$,
while its transfer function is a degree-$2k$ rational function.
Under some mild assumptions, transfer functions of the
form~\cref{eqn:MechTransfer} can be represented via a structured
barycentric form
\begin{equation} \label{eqn:MechBaryForm}
  \Hso(s) = \frac{\nso(s)}{\dso(s)} =
    \frac{\displaystyle \sum_{j = 1}^{k} \frac{h_{j} w_{j}}%
    {(s - \lambda_{j}) (s - \sigma_{j})}}%
    {\displaystyle 1 + \sum_{j = 1}^{k} \frac{w_{j}}%
    {(s - \lambda_{j})(s - \sigma_{j})}},
\end{equation}
with $\lambda_{j}, w_{j}, h_{j}, \sigma_{j} \in \C$ for $j = 1, \ldots, k$;
see~\cite{GosGW24}.
As before, the $\lambda_{j}$'s are referred to as
\emph{barycentric support points}, the $w_{j}$'s are the
\emph{barycentric weights}, and the $h_{j}$'s are \emph{function values}.
The additional parameters $\sigma_j$ are called the \emph{quasi-support points}.
The term quasi-support points has been introduced in~\cite{GosGW24} and
refers to the fact that~\cref{eqn:MechBaryForm} satisfies, additionally to 
$\Hso(\lambda_{j}) = h_{j}$, also $\Hso(\sigma_{j}) = h_{j}$,
provided that $w_{j} \neq 0$, for $j = 1, \ldots k$.
    
As was the case for the unstructured barycentric
form~\cref{eqn:UnstructuredBaryForm}, one may obtain a state-space
representation of the second-order system~\cref{eqn:MechSystem}
(and its transfer function~\cref{eqn:MechTransfer}) from the barycentric
representation~\cref{eqn:MechBaryForm}.
Recall the definitions of $\mLambda$, $\mw$, and $\mh$
in~\cref{eqn:Lam_w_h_def}, and additionally define
\begin{equation} \label{eqn:SIGMA_def}
  \mSigma = \mdiag(\sigma_{1}, \ldots, \sigma_{k}) \in \C^{k \times k}.
\end{equation}
Then, a matrix representation~\cref{eqn:MechTransfer} of the second-order barycentric form~\cref{eqn:MechBaryForm} can be obtained as
\begin{equation} \label{eqn:MDK_def}
  \mM = \eye{k}, \quad
  \mD = -\mLambda - \mSigma, \quad
  \mK = \mLambda \mSigma + \mw \ones{k}^{\trans}, \quad
  \mb = \mw  \quad
  \text{and}\quad
  \mc = \mh;
\end{equation}
see,~\cite{GosGW24} for further details.
Unlike the unstructured transfer function~\cref{eqn:UnstructuredTransfer}, not
every degree-$2k$ rational function can be written as a second-order transfer
function~\cref{eqn:MechTransfer}.
Thus, to recover second-order models from data, specialized methods such as
those that we develop in this work are required.

Note that in addition to~\cref{eqn:MechBaryForm},~\cite{GosGW24} provides
another barycentric representation for second-order
systems~\cref{eqn:MechSystem}.
This additional barycentric form imposes different assumptions on the
barycentric parameters, which complicates its use for data-driven modeling.
Therefore, we will not further investigate this additional barycentric form
in this work.


\subsection{The Variable Projection method}%
\label{sec:varpro}

The Variable Projection method~\cite{GolP73} (\VP) is a powerful
technique to accelerate the optimization of objective functions in linearly
separable form.
Since we will employ \VP{} in one of our proposed algorithms in
\Cref{sec:SecondOrderAAAs}, we will briefly summarize it here.

\VP{} considers optimization problems of the form
\begin{equation} \label{eqn:VarProMinimizationProblem}
  \min\limits_{\malpha, \mbeta \in \C^{n}}
    \lVert \mr(\malpha, \mbeta) \rVert_{2}^{2},
\end{equation}
with $\mr(\malpha, \mbeta) \in \C^{M}$ being defined entrywise
\begin{equation} \label{eqn:variableProjObjFun}
  \left[ \mr(\malpha, \mbeta) \right]_{i} = f_{i} -
    \sum\limits_{j = 1}^{k} \alpha_{j} \psi_{i, j}(\mbeta),
  \quad\text{for}~i = 1, \ldots, M,
\end{equation}
where $\mf = \begin{bmatrix} f_{1} & \ldots & f_{M} \end{bmatrix}^{\trans}
\in \C^{M}$ are function value samples, $\malpha, \mbeta \in \C^{k}$ are
parameters to be optimized, and each $\psi_{i, j}$ is a nonlinear function of
one or more components of $\mbeta$.
Defining the matrix $\mPsi(\mbeta) \in \C^{M \times k}$ entrywise as
\begin{equation}
  \left[ \mPsi(\mbeta) \right]_{i, j} = \psi_{i, j}(\mbeta)
\end{equation}
allows to express~\cref{eqn:VarProMinimizationProblem} equivalently via
\begin{equation} \label{eqn:variableProjObjFun_asMat}
  \min\limits_{\malpha,\mbeta}
    \left\lVert \mPsi(\mbeta) \malpha - \mf \right\rVert_{2}^{2}.
\end{equation}
For a fixed value of $\mbeta$, the optimal value of $\malpha$ is given by
$\mPsi(\mbeta)^{\dagger}\mf$, where $\mV^{\dagger}$ denotes the Moore-Penrose
pseudoinverse of the matrix $\mV$.
Substituting this optimal $\malpha$ into~\cref{eqn:variableProjObjFun_asMat}
yields the projected optimization problem
\begin{equation} \label{eqn:variableProjObjFun_projected}
  \min\limits_{\mbeta} \lVert \mPsi(\mbeta) \mPsi(\mbeta)^{\dagger} \mf -
    \mf \rVert_{2}^{2}.
\end{equation}
The projected problem~\cref{eqn:variableProjObjFun_projected} is now a
nonlinear least-squares problem depending only on the $k$ parameters in
$\mbeta$.
After the minimization of~\cref{eqn:variableProjObjFun_projected}, the optimal
nonlinear parameters,~$\mbeta_{\ast}$, have been identified.
The optimal linear parameters~$\malpha_{\ast}$ can easily be identified by
computing
\begin{equation} \label{eqn:optimalLinParmsGivenNonLin}
  \malpha_{\ast} = \mPsi(\mbeta_{\ast})^{\dagger} \mf.
\end{equation}

The optimization of~\cref{eqn:variableProjObjFun_projected} requires
differentiating the pseudoinverse of $\mPsi(\mbeta)$ with respect
to the optimization variables $\mbeta$.
In general, gradients of the cost
function~\cref{eqn:variableProjObjFun_projected} can be rather computationally
expensive so that a single step of an optimization solver
for~\cref{eqn:variableProjObjFun_projected} is typically more expensive
than a single step of a solver for the original
problem~\cref{eqn:VarProMinimizationProblem}.
However, it has been frequently observed that \VP{} converges in far fewer
iterations than optimization methods applied
to~\cref{eqn:VarProMinimizationProblem}.
In fact, in the case that the Gauss-Newton method is used, \VP{} is
guaranteed to converge faster than the direct approach~\cite{RuhW80}.
Furthermore, it has been observed in~\cite{Kro74} that \VP{} may converge when
direct optimization methods fail, which could be due to a better optimization
landscape coming from the projected cost function.

Since the introduction of \VP{} in~\cite{GolP73}, there have been many
advances of the method; see, for example, the work on derivative
calculations~\cite{Kau75} and the implementational details in~\cite{OLeR13}.
For a survey on \VP{} and its applications, see~\cite{GolP03}.


\section{Revisiting the unstructured AAA algorithm}%
\label{sec:AAA_Unstructured}

The methods that we propose in this work for structured data-driven modeling
are strongly inspired by the Adaptive-Anderson-Antoulas (AAA)
algorithm~\cite{NakST18}.
To this end, we recap the foundations and basic functionality of the AAA
algorithm in this section.

Let $\FOM\colon \C \to \C$ denote an arbitrary unknown function for which the
following data set is given
\begin{equation} \label{eqn:M_def}
  \M \coloneqq \{ (\mu_{1}, g_{1}, \eta_{1}), \ldots,
    (\mu_{M}, g_{M}, \eta_{M}) \},
\end{equation}
where $\mu_{i} \in \C$ are \emph{evaluation points},
$g_{i} = G(\mu_{i}) \in \C$ are the \emph{function values} corresponding to the
evaluation points, and $\eta_{i} > 0$ are optional \emph{weights}.
We aim to find a degree-$k$ strictly proper rational function in barycentric
form~\cref{eqn:UnstructuredBaryForm} that approximates the given data
in~\cref{eqn:M_def} well.
Due to the interpolatory properties of the barycentric
form~\cref{eqn:UnstructuredBaryForm}, we may choose support points
$\{ \lambda_{j} \}_{j = 1}^{k} \subseteq \{ \mu_{i} \}_{i = 1}^{M}$ from
$\M$ at which to interpolate the corresponding function values
$\{ h_{j} = G(\lambda_{j}) \}_{j = 1}^{k} \subseteq \{ g_{i} \}_{i = 1}^{M}$
by the construction of the barycentric form.
We gather this interpolation data in a set of tuples
\begin{equation}
  \Pk = \{ (\lambda_{1}, h_{1}), \ldots, (\lambda_{k}, h_{k}) \}.
\end{equation}
Once $\Pk$ has been determined, the \emph{barycentric weights} 
\begin{equation} \label{eqn:wkVec_def}
  \mwk = \begin{bmatrix} \wk_{1} & \ldots & \wk_{k} \end{bmatrix}^{\trans}
\end{equation}
are chosen to improve the fit of the function approximation with respect to the
remaining data in $\M$.  
The barycentric weights $\mwk$ can be chosen to enforce additional
interpolation conditions, as is done in~\cite{AntA86}.
In AAA, the weights $\mwk$ are used instead to minimize a linearized
least-squares residual on the uninterpolated data.

Here, we rederive AAA with several modifications from its original
presentation in~\cite{NakST18}.
These modifications adapt the algorithm to be a suitable foundation for our
proposed second-order AAA algorithms, and the adapted \AAA{} will be more
appropriate for numerical comparisons later on.
Specifically, we introduce a different initialization strategy, enforce
strictly proper rational functions, and explicitly incorporate weighted data
fitting.
From here on, we denote the \emph{unstructured} AAA algorithm by \AAA{}, since,
in contrast to the methods we develop in \Cref{sec:SecondOrderAAAs}, it does
not enforce any structure on the rational functions it generates.
    
Define the initial rational function and data sets to be
\begin{equation}
  \Hr^{(0)}(s) = 0, \quad \M^{(0)} = \M \quad\text{and}\quad
  \Pset^{(0)} = \emptyset.
\end{equation}
Then, in iteration $k > 0$, \AAA{} first determines the tuple
$(\mu, g, \eta) \in \M^{(k - 1)}$ that maximizes the (nonlinear) weighted
approximation error
\begin{equation} \label{eqn:NextLamIsWorstMu_USAAA}
  (\mu, g, \eta) = \argmax\limits_{(\mu_{i}, g_{i}, \eta_{i}) \in
    \M^{(k - 1)}}\eta_{i} \lvert \Hr^{(k - 1)} (\mu_{i}) - g_{i} \vert.
\end{equation}
The selected data sample is then added to the interpolation data set
$\Pk = \Pset^{(k - 1)} \cup \{ (\lambda_{k}, h_{k}) \}$, where the evaluation
point $\mu$ is relabeled as $\lambda_{k}$ and the corresponding function value
$g$ as $h_{k}$.
The tuple is then removed from the test set $\Mk = \M^{(k - 1)} \setminus
\{(\mu, g, \eta)\}$.
Thus, at iteration $k$, the set $\Mk$ contains $M - k$ tuples and
$\Pk$ contains $k$ tuples.
After the removal of $\{(\mu, g, \eta)\}$, the indices of elements in $\Mk$ are
relabeled to $1, \ldots, M - k$, that is, there are no gaps in the index set.

Note that the maximizer of~\cref{eqn:NextLamIsWorstMu_USAAA} might not be
unique.
For example, in the case of relative data weighting
(i.e., $\eta_{i} = 1 / \lvert g_{i} \rvert$), for $k = 1$, approximation error
in~\cref{eqn:NextLamIsWorstMu_USAAA} satisfies
$\eta_{i} \lvert \Hr^{(0)}(\mu_{i}) - g_{i} \vert = 1$ for all
$i = 1, \ldots, M$.
Assume that at iteration $k$ there is a non-singleton
$\M^{(k - 1)}_{\max} \subset \M^{(k-1)}$ solution
to~\cref{eqn:NextLamIsWorstMu_USAAA}.
To remove the ambiguity, we propose to continue the selection of a single
error maximizer by selecting the point that maximizes the unweighted
approximation error from the set $\M_{\max}^{(k - 1)}$, i.e.,
\begin{equation} \label{eqn:NextLamIsWorstMu_unweighted_USAAA}
  (\mu, g, \eta) = \argmax\limits_{(\mu_{i}, g_{i}, \eta_{i}) \in
    \M_{\max}^{(k - 1)}} \rvert \Hr^{(k - 1)}(\mu_{i}) - g_{i} \lvert.
\end{equation}
If there are still multiple maximizers
in~\cref{eqn:NextLamIsWorstMu_unweighted_USAAA}, we fall back on choosing the
maximizer with the smallest index.

After the selection of $\lambda_{k}$ and $h_{k}$ and the update of the sets
$\Mk$ and $\Pk$, the barycentric weights $\mwk$ must be determined to yield
good approximations with respect to the uninterpolated data in $\Mk$.
Ideally, we would like to solve
\begin{equation} \label{eqn:TrueLSProblem_unstruct}
    \min\limits_{\mwk} \lVert \ekTrue \rVert_{2}^{2},
\end{equation}
where $\ekTrue \in \C^{M - k}$ is defined as
\begin{subequations} \label{eqn:ektrueDef_unstructured}
\begin{align}
  \ekTrue_{i} = \eta_{i} \left( \Hrk(\mu_{i}) - g_{i} \right) & =
    \eta_{i} \left( \frac{\numerUSk(\mu_{i})}{\denomUSk(\mu_{i})} - g_{i}
    \right) \\
  & =  \eta_{i} \left( \left( \frac{\sum_{j = 1}^{k}
    \frac{h_{j} w_{j}}{\mu_{i} - \lambda_{j}}}%
    {1 + \sum_{j = 1}^{k}
    \frac{w_{j}}{\mu_{i} - \lambda_{j}}} \right) - g_{i} \right),
\end{align}
\end{subequations}
for $(\mu_{i}, g_{i}, \eta_{i}) \in \Mk$ and $i = 1, \ldots, M - k$.
Since~\cref{eqn:ektrueDef_unstructured} is nonlinear in the barycentric weights
$w_{j}$, \AAA{} instead considers the linearized weighted residual vector
$\ek \in \C^{M - k}$ whose $i$-th entry is
\begin{subequations} \label{eqn:linearizedError_unstruct}
\begin{align}
  \ek_{i} & = \eta_{i} (\numerUSk(\mu_{i}) - g_{i} \denomUSk(\mu_{i})) \\
  & = \eta_{i} \left( \left( \sum_{j = 1}^{k} \frac{h_{j} \wk_{j}}%
    {\mu_{i} - \lambda_{j}} \right) - g_{i} \left(1 + \sum_{j = 1}^{k}
    \frac{\wk_{j}}{\mu_{i} - \lambda_{j}} \right) \right) \\
  & = \eta_{i} \left( \left( \sum_{j = 1}^{k} \wk_{j} \frac{h_{j} - g_{i}}%
    {\mu_{i} - \lambda_{j}} \right) - g_{i} \right)
\end{align}
\end{subequations}
for $(\mu_{i}, g_{i}, \eta_{i}) \in \Mk$ and $i = 1, \ldots, M - k$.
Define the Loewner matrix $\LowUSk \in \C^{(M - k) \times k}$ entrywise via 
\begin{equation} \label{eqn:LoewnerMat_USAAA}
  \LowUSk_{i, j} = \frac{g_{i} - h_{j}}{\mu_{i} - \lambda_{j}},
    \quad \text{where}\quad
    (\mu_{i}, g_{i}, \eta_{i}) \in \Mk
    \quad\text{and}\quad
    (\lambda_{j}, h_{j}) \in \Pk,
\end{equation}
for $i = 1, \ldots, M - k$ and $j = 1, \ldots k$.
Also, define the vectors $\mgk \in \C^{M - k}$ and $\ETAk \in \R^{M-k}$
entrywise to be
\begin{equation} \label{eqn:GkAndEtakDef}
  \mgk_{i} = g_{i} \quad\text{and}\quad
  \ETAk_{i} = \eta_{i}, \quad\text{where}\quad
  (\mu_{i}, g_{i}, \eta_{i}) \in \Mk,
\end{equation}
for $i = 1, \ldots, M - k$.
Then, the desired barycentric weights $\mwk$ are computed as the solution to
the linear least-squares problem
\begin{equation} \label{eqn:LinearLsProblem_unstruct}
  \mwk = \argmin\limits_{\mw \in \C^{k}} \left\lVert
    \mdiag(\ETAk) \left( -\LowUSk \mw - \mgk \right) \right\rVert_{2}^{2}.
\end{equation}
Note that we keep the negative signs in~\cref{eqn:LinearLsProblem_unstruct} for
the clarity of presentation of the linear least-squares problem.
The vector of barycentric weights $\mwk$ that
solves~\cref{eqn:LinearLsProblem_unstruct} is given by
\begin{equation} \label{eqn:WeightedLS_USAAA}
  \mwk = -\left(\mdiag(\ETAk) \LowUSk \right)^{\dagger}
    \left( \mdiag(\ETAk) \mgk \right).
\end{equation}

This process of greedily choosing the next support point $\lambda_{k}$ and
function value $h_{k}$ via~\cref{eqn:NextLamIsWorstMu_USAAA}, then fitting the
barycentric weights $\mwk$ by solving the linearized least-squares
problem~\cref{eqn:LinearLsProblem_unstruct} is iteratively repeated for
$k = 1, \ldots, \maxOrder$.
Here, the number $\maxOrder$ is the maximum number of iteration steps as well as
the maximum order of the associated model.
Optionally, a convergence tolerance for approximation error measures based on
the nonlinear residual $\ekTrue$ in~\cref{eqn:linearizedError_unstruct} can
allow the iteration to terminate early when the approximation is accurate
enough.
Such measures are typically weighted norms of $\ekTrue$; see, for example, the
setup of the numerical experiments in \Cref{sec:numerics}.
We summarize the unstructured AAA algorithm in \Cref{alg:UnstructAAA}.
System matrices $\mE, \mA, \mb, \mc$ corresponding to the output of
\Cref{alg:UnstructAAA} can be directly obtained using~\cref{eqn:EAbc_def}.

\begin{algorithm}[t]
  \SetAlgoHangIndent{1pt}
  \DontPrintSemicolon
  \caption{Unstructured AAA algorithm (\AAA).}%
  \label{alg:UnstructAAA}

  \KwIn{Data set $\M = \{(\mu_{i}, g_{i}, \eta_{i})\}_{i = 1}^{M}$ and
    maximum model order $\maxOrder$.}
  \KwOut{Parameters of the barycentric form
    $\lambda_{j}, h_{j}, w_{j}$, for~$j = 1, \ldots, k$.}
  
  \For{$k = 1, \ldots, \maxOrder$}{
    Find $(\mu, g, \eta) \in \M^{(k - 1)}$ that maximizes the
      weighted approximation error
      \begin{equation*}
        (\mu, g, \eta) = \argmax\limits_{(\mu_{i}, g_{i}, \eta_{i}) \in
          \M^{(k - 1)}} \eta_{i} \lvert \Hr^{(k - 1)}(\mu_{i}) - g_{i} \rvert.
      \end{equation*}\;
      \vspace{-\baselineskip}
      
    Update the barycentric parameters and the data sets
      \begin{equation*}
        \lambda_{k} = \mu, \quad
          h_{k} = g, \quad
          \Mk = \M^{(k - 1)} \setminus \{(\mu, g, \eta)\}, \quad
          \Pk = \Pset^{(k - 1)} \cup \{(\lambda_{k}, h_{k})\}.
      \end{equation*}\;
      \vspace{-\baselineskip}
    
    Solve the linearized least-squares problem
      \begin{equation*}
        \mwk = \argmin\limits_{\mw \in \C^{k}} \left\lVert \mdiag(\ETAk)
          \left(-\LowUSk \mw - \mgk \right) \right\rVert_{2}^{2}
      \end{equation*}
      via~\cref{eqn:WeightedLS_USAAA} for the barycentric weights.\;
  }
\end{algorithm}


\section{Second-order AAA algorithms}
\label{sec:SecondOrderAAAs}

The (unstructured) AAA algorithm presented in \Cref{sec:AAA_Unstructured} makes
use of the interpolatory barycentric form~\cref{eqn:UnstructuredBaryForm} and
the linearized residual~\cref{eqn:linearizedError_unstruct} to
combine rational function interpolation with least-squares approximation.
Our goal in this section is to develop second-order AAA algorithms, that is,
algorithms that follow the same two-step procedure of greedy interpolation
and least-squares minimization as \AAA{} but produce second-order systems of
the form~\cref{eqn:MechSystem}.
The remainder of this section is organized as follows:
\Cref{sec:settingStage} discusses the general structure of our algorithms and
presents a nonlinear least-squares problem similar
to~\cref{eqn:TrueLSProblem_unstruct}.
Then, \Cref{sec:SOAAA,sec:SONLAAA,sec:LinSOAAA} discuss three approaches for
solving this optimization problem, each of which lead to a different
second-order AAA algorithm.


\subsection{Algorithmic foundations}%
\label{sec:settingStage}

As in \Cref{sec:AAA_Unstructured}, we assume access to the data
\begin{equation} \label{eqn:M_def2}
  \M \coloneq \{ (\mu_{1}, g_{1}, \eta_{1}), \ldots,
    (\mu_{M}, g_{M}, \eta_{M})\}.
\end{equation}
However, this time we seek to determine the parameters of a second-order
barycentric model
\begin{equation}
  \HrkSO(s) = \frac{\numerSOk(s)}{\denomSOk(s)} =
    \frac{\displaystyle \sum_{j = 1}^{k} \frac{h_{j} \wk_{j}}%
    {(s - \lambda_{j}) (s - \sigk_{j})}}%
    {\displaystyle 1 + \sum\limits_{j = 1}^{k} \frac{\wk_{j}}%
    {(s - \lambda_{j}) (s - \sigk_{j})}}
\end{equation}
to fit the data in $\M$.
Like the unstructured barycentric form~\cref{eqn:UnstructuredBaryForm},
the second-order barycentric from~\cref{eqn:MechBaryForm} has
the interpolation property $\HrkSO(\lambda_{j}) = h_{j}$,
for $j = 1, \ldots, k$.
Thus, we will greedily choose the barycentric support points
$\{ \lambda_{j} \}_{j = 1}^{k}$ from $\M$ at which to interpolate the
corresponding function values $\{ h_{j} \}_{j = 1}^{k}$.
As before, we collect these parameters into the set
\begin{equation}
  \PkSO = \{ (\lambda_{1}, h_{1}, \sigk_{1}), \ldots,
    (\lambda_{k}, h_{k},\sigk_{k}) \},
\end{equation}
where the $\{ \sigk_{j} \}_{j = 1}^{k}$ are \emph{initializations of the
quasi-support points}.
At step $k$ of the algorithm, after the support point $\lambda_{k}$ and
function value $h_{k}$ have been chosen according to the greedy procedure
described in \Cref{sec:AAA_Unstructured}, and the quasi-support points
\begin{equation} \label{eqn:sigkVec_def}
  \sigkVec = \begin{bmatrix} \sigk_{1} & \ldots & \sigk_{k}
    \end{bmatrix}^{\trans} \in \C^{k}
\end{equation}
have been initialized, we then seek to update the quasi-support points
$\sigkVec$ as well as the barycentric weight vector
\begin{equation}
  \mwk = \begin{bmatrix} \wk_{1} & \ldots & \wk_{k} \end{bmatrix}^{\trans}
    \in \C^{k}
\end{equation}
to solve the nonlinear least-squares problem
\begin{equation} \label{eqn:TrueLSProblem_SO}
  \min\limits_{\mwk, \sigkVec} \left\lVert \ekTrueSO \right\rVert_{2}^{2},
\end{equation}
where $\ekTrueSO \in \C^{M - k}$ is the residual vector defined entrywise
defined via
\begin{equation} \label{eqn:err_trueSO}
  \left[ \ekTrueSO \right]_{i} \coloneqq \eta_{i} \left(
    \frac{\numerSOk(\mu_{i})}{\denomSOk(\mu_{i})} - g_{i} \right) =
    \eta_{i} \left( \left( \frac{\sum_{j = 1}^{k} \frac{h_{j} \wk_{j}}%
    {(\mu_{i} - \lambda_{j}) (\mu_{i} - \sigk_{j})}}%
    {1 + \sum_{j = 1}^{k} \frac{\wk_{j}}{(\mu_{i} - \lambda_{j})
    (\mu_{i} - \sigk_{j})}} \right) - g_{i} \right),
\end{equation}
for $(\mu_{i}, g_{i}, \eta_{i})\in \Mk$ and $i = 1, \ldots, M - k$.

It is important to note that, in general, we cannot use the $\sigk$ parameters
to enforce additional interpolation conditions.
This is because not only does the barycentric form~\cref{eqn:MechBaryForm}
ensure that $\HrkSO(\lambda_{j}) = h_{j}$ but also $\HrkSO(\sigk_{j}) = h_{j}$.
Thus, unless there are two data tuples $(\mu_{p}, g_{p}, \eta_{p}) \in \M$ and
$(\mu_{q}, g_{q}, \eta_{q}) \in \M$ where $g_{p} = g_{q}$ while
$\mu_{p} \neq \mu_{q}$, the $\sigk$ parameters cannot be used to enforce
additional interpolation conditions.
In the following, we discuss three different methods to solve the nonlinear
least-squares problem~\cref{eqn:TrueLSProblem_SO}, leading to three different
second-order variants of the AAA algorithm. 


\subsection{Simplification to separable residual vector}%
\label{sec:SOAAA}

The residual~\cref{eqn:err_trueSO} is nonlinear in both the barycentric weights
$\wk_{j}$ and the quasi-support points $\sigk_{j}$.
Similar to the unstructured AAA algorithm in \Cref{sec:AAA_Unstructured}, we
seek to find a residual expression that is easier to optimize than the true
nonlinear residual~\cref{eqn:err_trueSO}.
We apply the same idea as in \Cref{sec:AAA_Unstructured} to obtain
the simplified residual $\ekSO \in \C^{M-k}$, that is, we factor out and remove
the denominator of the barycentric form.
Entrywise, this residual vector is given as
\begin{subequations} \label{eqn:separableObjFun}
\begin{align}
  \left[ \ekSO \right]_{i} & =
    \eta_{i} \left( \numerSOk(\mu_{i}) - g_{i} \denomSOk(\mu_{i}) \right)\\
  & = \eta_{i} \left( \left( \sum\limits_{j = 1}^{k}
    \frac{h_{j} \wk_{j}}{(\mu_{i} - \lambda_{j}) (\mu_{i} - \sigk_{j})} \right)
    - g_{i} \left(1 + \sum\limits_{j = 1}^{k}
    \frac{\wk_{j}}{(\mu_{i} - \lambda_{j}) (\mu_{i} - \sigk_{j})} \right)
    \right)\\
  & = \eta_{i} \left( \left( \sum\limits_{j = 1}^{k} \wk_{j}
    \frac{h_{j} - g_{i}}{(\mu_{i} - \lambda_{j}) (\mu_{i} - \sigk_{j})}
    \right) - g_{i} \right),
\end{align}
\end{subequations}
for $i = 1, \ldots, M - k$.
Note that while $\ek$ in \Cref{sec:AAA_Unstructured} was linear in its unknowns
after this simplification, the structured residual $\ekSO$ is still nonlinear.
More specifically, the residual $\ekSO$ has a \emph{separable} nonlinearity in
the sense that each entry of $\ekSO$ is a linear combination of nonlinear
functions of the quasi-support points $\sigk_{j}$ with the barycentric weights
as linear coefficients $\wk_{j}$.
Due to this particular structure in~\cref{eqn:separableObjFun}, we seek to use
\VP{} for the optimization.
Following \Cref{sec:varpro}, we define the matrix of nonlinear functions of
$\sigk_{j}$ as the Loewner-like matrix $\LowSOk \in \C^{(M - k)\times k}$,
with the entries
\begin{equation} \label{eqn:LoewnerMat_SO}
  \left[ \LowSOk \right]_{i, j} = \frac{g_{i} - h_{j}}%
    {(\mu_{i} - \lambda_{j}) (\mu_{i} - \sigk_{j})},
\end{equation}
where $(\mu_{i}, g_{i}, \eta_{i}) \in \Mk$ and
$(\lambda_{j}, h_{j}, \sigk_{j}) \in \PkSO$,
for $i = 1, \ldots, M - k$ and $j = 1, \ldots, k$.
Defining $\mgk$ and $\ETAk$ as in~\cref{eqn:GkAndEtakDef}, we aim to solve
\begin{equation} \label{eqn:SeparableLSProblem}
  \min\limits_{\mwk, \sigkVec} \lVert \ekSO \rVert_{2}^{2} =
    \min\limits_{\mwk, \sigkVec} \left\lVert \mdiag(\ETAk)
      \left(-\LowSOk \mwk - \mgk \right) \right\rVert_{2}^{2}
\end{equation}
as a surrogate for~\cref{eqn:TrueLSProblem_SO}.
While not fully linear, we expect the separable
residual~\cref{eqn:separableObjFun} to be easier to optimize
than~\cref{eqn:err_trueSO}, especially when $k$ (the current model size) is
large.
Solving the optimization problem with \VP{} will further decrease computation
times.

Successful optimization of~\cref{eqn:SeparableLSProblem} depends on a good
initialization of the $\sigk_{j}$ parameters.
At iteration $k$, the previous $\sigma^{(k - 1)}_{j}$,
for $j = 1,\ldots, k - 1$, parameters have already been optimized.
Thus, it is suitable to set $\sigk_{j} = \sigma^{(k - 1)}_{j}$ for
$j = 1,\ldots, k-1$.
In this case, an initialization strategy is needed only for the remaining
$\sigk_{k}$.
As remarked in~\cite{GosGW24}, since our model $\HrkSO$ satisfies
$\HrkSO(\sigk_{k}) = h_{k}$, initializing $\sigk_{k}$ near another sample point
$\mu_{p}$ could lead to large approximation errors if $h_{p} \neq h_{k}$. 
Thus, $\sigk_{k}$ should be initialized far from the given sample points.
In typical applications, this means far from the imaginary axis.
The heuristics in~\cite{GosGW24} suggest to choose $\sigk_{k}$ far inside the
left half-plane.
We observed that this selection causes the objective
function~\cref{eqn:SeparableLSProblem}
to decrease monotonically as the order of approximation $k$ increases,
as desired.
In particular, we initialize $\sigk_k$ as 
\begin{equation}
  \sigk_{k} = c - \imunit\imag(\lambda_k),
\end{equation}
where $c$ is a large negative real number, $\imunit \coloneq \sqrt{-1}$ is
the imaginary unit and $\imag(\lambda_{k})$ is the imaginary part of the
barycentric support point $\lambda_{k}$.
In \Cref{sec:AccuracyHeuristics}, we provide a theoretical analysis that
further supports this initialization strategy.

For completeness, we supply the Jacobian of the residual vector $\ekSO$
in~\cref{eqn:separableObjFun} with respect to $\sigkVec$ that will be used
in~\VP.
Conveniently, this quantity can be computed easily from the Loewner-like
matrix~\cref{eqn:LoewnerMat_SO} entrywise as
\begin{equation} \label{eqn:Jacobian_SO}
  \left[ \JacSOk(\sigkVec) \right]_{i, j} =
    -\left[ \LowSOk \right]_{i, j} \frac{\eta_{i}}{\mu_{i} - \sigk_{j}},
\end{equation}
for $i = 1, \ldots, M - k$ and $j = 1, \ldots, k$.
The presence of $\LowSOk$ in the Jacobian computation offers a great advantage
in terms of computation costs:
We must compute $\LowSOk$ for function evaluations so that reusing it for
Jacobian computations yields significant savings.

The resulting separable second-order AAA algorithm (\AAAFull) is summarized in
\Cref{alg:SOAAA}.
Second-order system matrices $\mM, \mD, \mK,\mb,\mc$ corresponding to the
output of \Cref{alg:SOAAA} can be directly obtained using~\cref{eqn:MDK_def}.

\begin{algorithm}[t]
  \SetAlgoHangIndent{1pt}
  \DontPrintSemicolon
  \caption{Separable second-order AAA algorithm (\AAAFull).}%
  \label{alg:SOAAA}

  \KwIn{Data set $\M = \{(\mu_{i}, g_{i}, \eta_{i})\}_{i = 1}^{M}$ and
    maximum model order $\maxOrder$.}
  \KwOut{Parameters of the second-order barycentric form
    $\lambda_{j}, \sigma_{j}, h_{j}, w_{j}$, for~$j = 1, \ldots, k$.}
  
  \For{$k = 1, \ldots, \maxOrder$}{
    Find $(\mu, g, \eta) \in \M^{(k - 1)}$ that maximizes the weighted
      approximation error
      \begin{equation*}
        (\mu, g, \eta) = \argmax\limits_{(\mu_{i}, g_{i}, \eta_{i}) \in
          \M^{(k - 1)}} \eta_{i} \lvert \HrSO^{(k - 1)}(\mu_{i}) - g_{i} \rvert.
      \end{equation*}\;
      \vspace{-\baselineskip}

    Update the barycentric parameters
      \begin{equation*}
          \lambda_{k} = \mu, \quad
          h_{k} = g, \quad
          \sigkVec = \begin{bmatrix} \msigma^{(k-1)} \\
            c - \imunit\imag(\lambda_{k}) \end{bmatrix}
      \end{equation*}
      and the data sets
      \begin{equation*}
        \Mk = \M^{(k - 1)} \setminus \{(\mu, g, \eta)\}, \quad
        \PkSO = \PSO^{(k - 1)} \cup \{(\lambda_{k}, h_{k}, \sigk)\}.
      \end{equation*}\;
      \vspace{-\baselineskip}
      
    Solve the separable nonlinear least-squares problem
      \begin{equation*}
        \min\limits_{\mwk, \sigkVec} \left\lVert \mdiag(\ETAk)
          \left( -\LowSOk \mwk - \mgk \right) \right\rVert_{2}^{2}
      \end{equation*}
      using \VP{} for the barycentric weights $\mwk$ and quasi-support points
      $\sigkVec$.\;
  }
\end{algorithm}

\begin{remark}
  Note that~\cref{eqn:SeparableLSProblem} requires optimization of a real
  function that depends on complex parameters, which is necessarily
  non-analytic.
  Hence, \Cref{eqn:SeparableLSProblem} cannot be optimized in a traditional
  sense, even though the Jacobian~\cref{eqn:Jacobian_SO} of the residual
  vector~\cref{eqn:separableObjFun} is readily computable.
  This issue can be circumvented by expressing the optimization parameters in
  terms of their real and complex parts, and then minimizing these as
  two sets of real parameters.
  The downside of this approach is that the Jacobians of these objective
  functions typically lose the convenient structures present in the complex
  Jacobians.
  In this work, we instead rely on the
  \emph{Wirtinger calculus}~\cite{Wir27} (or $\C\R$-calculus) and the
  so-called \emph{complex gradient operator}.
  A nonanalytic function $f(\mz): \C^n \to \R$ can be optimized with the
  Wirtinger calculus if it can be written as $\hat{f}(\mz, \overline{\mz})$
  that is analytic in $\mz$ when $\overline{\mz}$ is regarded as a
  constant and analytic in $\overline{\mz}$ when $\mz$ is regarded as a
  constant.
  For more information, we refer to~\cite{KooQKetal23, Kre09, SorVD12}.
\end{remark}


\subsection{Other second-order AAA methods}%
\label{sec:othermethods}

For most problems of interest, we recommend using \AAAFull{} from the previous
\Cref{sec:SOAAA} to construct second-order dynamical systems from data.
The separable form of the residual vector~\cref{eqn:separableObjFun} can be
efficiently optimized with \VP{}, and offers typically good approximation
quality.
However, in some applications such as those involving real-time data collection
like digital twins~\cite{TaoZZ24, WilS24}, models must be created or updated in
fractions of a second.
In such scenarios, even the speed up offered by \VP{} may not be sufficient.  
For such cases, we introduce the Fully Linearized Second-order AAA Algorithm
(\AAALin) in \Cref{sec:LinSOAAA}.
This method avoids the solution of nonlinear optimization problems altogether
for the price of approximation accuracy.

On the other hand, in some applications such as control system
design~\cite{GilGPetal17, Sax19}, the final size of the data-driven model
must be as small and as accurate as possible, while the computation time
used for constructing the model is not as important.
While very competitive at moderate to large reduced orders, we find that the
separable objective function considered in \Cref{sec:SOAAA} may limit
the accuracy of the models for very small sizes $k$, and thus may be less
suitable for these applications.
For such cases, we introduce the Fully Nonlinear Second-order AAA Algorithm
(\AAAFullNL) in \Cref{sec:SONLAAA}.
Both \AAALin{} and \AAAFullNL{} use the greedy update selection described in
\Cref{sec:AAA_Unstructured} to choose their barycentric support points.


\subsubsection{Linearized second-order AAA}%
\label{sec:LinSOAAA}

The separable objective function~\cref{eqn:separableObjFun} is linear in the
barycentric weights $\wk_{j}$ but nonlinear in the quasi-support points
$\sigk_{j}$.
As discussed in \Cref{sec:SOAAA}, placing the $\sigk_{j}$ parameters close to
the sample points $\mu_{i}$ is likely to lead to approximation errors, while
placing the $\sigk_{j}$'s far in the left half plane leads to monotonic
decay of the simplified residual $\ekSO$ in~\cref{eqn:SeparableLSProblem}.
This is true even if $\sigkVec$ is not updated after its initialization.
Thus, we may choose to leave the quasi-support points $\sigk_{j}$ where they
have been initialized in \AAAFull{} to obtain a fully linearized second-order AAA
method.

The fully linearized second-order AAA algorithm (\AAALin) then only requires
single linear least-squares solve in each iteration $k$, and as such has speed
comparable to \AAA.
Specifically, define the Loewner-like matrix $\LowSOk$ as
in~\cref{eqn:LoewnerMat_SO}.
Then, with the vector of function values $\mgk \in \C^{(M-k)}$  and the
data weights $\ETAk \in \R^{(M-k)}$ as in~\cref{eqn:GkAndEtakDef}, we seek to
solve
\begin{equation} \label{eqn:LinearLSProblem}
  \min\limits_{\mwk} \lVert \ekSO \rVert_{2}^{2} =
    \min\limits_{\mwk} \left\lVert \mdiag(\ETAk) \left( -\LowSOk\mwk - \mgk
    \right) \right\rVert_{2}^{2}.
\end{equation}
The unique vector of barycentric weights $\mwk$ that minimizes
$\lVert \ekSO \rVert$ is given by
\begin{equation} \label{eqn:LinearLSSolution_SOLAAA}
  \mwk = -\left( \mdiag(\ETAk) \LowSOk \right)^{\dagger}
    \left( \mdiag(\ETAk) \mgk \right).
\end{equation} 
The resulting fully linearized second-order AAA algorithm, \AAALin{}, is
summarized in \Cref{alg:LinAAA}.
Similar to the previous second-order AAA method, \AAALin{} computes the
barycentric parameters so that we can obtain matrices of the corresponding
second-order system realization using~\cref{eqn:MDK_def}.

\begin{algorithm}[t]
  \SetAlgoHangIndent{1pt}
  \DontPrintSemicolon
  \caption{Linearized second-order AAA algorithm (\AAALin).}%
  \label{alg:LinAAA}

  \KwIn{Data set $\M = \{(\mu_{i}, g_{i}, \eta_{i})\}_{i = 1}^{M}$ and
    maximum model order $\maxOrder$.}
  \KwOut{Parameters of the second-order barycentric form
    $\lambda_{j}, \sigma_{j}, h_{j}, w_{j}$, for~$j = 1, \ldots, k$.}
  
  \For{$k = 1, \ldots, \maxOrder$}{
    Find $(\mu, g, \eta) \in \M^{(k - 1)}$ that maximizes the weighted
      approximation error
      \begin{equation*}
        (\mu, g, \eta) = \argmax\limits_{(\mu_{i}, g_{i}, \eta_{i})
          \in \M^{(k - 1)}} \eta_{i} \lvert \HrSO^{(k - 1)}(\mu_{i}) - g_{i}
          \rvert.
      \end{equation*}\;
      \vspace{-\baselineskip}

    Update the barycentric parameters
      \begin{equation*}
          \lambda_{k} = \mu, \quad
          h_{k} = g, \quad
          \sigkVec = \begin{bmatrix} \msigma^{(k-1)} \\
            c - \imunit\imag(\lambda_{k}) \end{bmatrix}
      \end{equation*}
      and the data sets
      \begin{equation*}
        \Mk = \M^{(k - 1)} \setminus \{(\mu, g, \eta)\}, \quad
        \PkSO = \PSO^{(k - 1)} \cup \{(\lambda_{k}, h_{k}, \sigk)\}.
      \end{equation*}\;
      \vspace{-\baselineskip}
    
    Solve the linearized least-squares problem
      \begin{equation*}
         \mwk = \argmin\limits_{\mw} \left\lVert \mdiag(\ETAk)
           \left( -\LowSOk \mw - \mgk \right) \right\rVert_{2}^{2}
      \end{equation*}
      via~\cref{eqn:LinearLSSolution_SOLAAA} for the barycentric weights.\;
  }
\end{algorithm}

We emphasize that \AAALin{} is meant for cases where computation time is the
most important factor.
Keeping the quasi-support points $\sigma_{j}$ fixed leads to much faster
computation times, with the possible trade-off of reduced approximation quality.


\subsubsection{Fully nonlinear second-order AAA Algorithm}%
\label{sec:SONLAAA}

Finally, in this section, we consider a variant of the second-order AAA
algorithm with the fully nonlinear least-squares problem
\begin{equation} \label{eqn:nonlinLS}
  \min\limits_{\mwk, \sigkVec} \left\lVert \ekTrueSO \right\rVert_{2}^{2},
\end{equation}
with the elements of the residual vector  $\ekTrueSO \in \C^{M - k}$ being
defined as
\begin{equation} \label{eqn:nonlinError}
  \left[ \ekTrueSO \right]_{i} =
    \eta_{i} \lvert \HrkSO(\mu_{i}) - g_{i} \rvert =
    \eta_{i} \left\lvert \frac{\sum_{j = 1}^{k} \frac{h_{j} \wk_{j}}%
    {(\mu_{i} - \lambda_{j}) (\mu_{i} - \sigk_{j})}}{1 + \sum_{j = 1}^{k}
    \frac{\wk_{j}}{(\mu_{i} - \lambda_{j}) (\mu_{i} - \sigk_{j})}} -
    g_{i}\right\rvert,
\end{equation}
where $(\mu_{i}, g_{i}, \eta_{i}) \in \Mk$ and for $i = 1, \ldots, M - k$.
In contrast to the separable residual~\cref{eqn:separableObjFun}, the
expression~\cref{eqn:nonlinError} is nonlinear in both the barycentric weights
$\wk_{j}$ and quasi-support points~$\sigk_{j}$ leading to a more difficult
optimization problem.
However, in the cases where particularly high approximation accuracy is
required, the optimization of~\cref{eqn:nonlinError} will provide the best
results.
The rest of this section describes \AAAFullNL{}, our second-order AAA method
for optimizing $\ekTrueSO$.

To optimize $\ekTrueSO$, we initialize $\sigk_{k}$ in the same manner as we did
for the separable objective function $\ekSO$ in \Cref{sec:SOAAA}.
Since the barycentric weights $\wk_{j}$ will now also be optimized by a
nonlinear solver, they too require an initialization strategy.
At iteration $k$, the previous weights $w_{j}^{(k-1)}$, for $j = 1,\ldots,k-1$,
have already been optimized.
Hence, we initialize $\wk_{j} = w^{(k-1)}_{j}$, for $j = 1, \ldots, k - 1$.
As was our strategy with the initialization of $\sigk_{k}$, we aim to promote a
monotonic decrease in the least-squares residual
$\lVert \ekTrueSO \rVert_{2}^{2}$ for increasing model size $k$.
One idea could be to initialize the remaining barycentric weight with
$\wk_{k} = 0$ so that the optimization starts with a function equal to the
previous approximation $\HrSO^{(k-1)}$.
However, due to the interpolation property of the barycentric form, the
objective function is not differentiable when $\wk_{j} = 0$ for any $j$.
Instead, we choose to initialize $\wk_{k}$ with a small absolute value to
minimize its impact at the start of the optimization.
In our numerical experiments, we have observed that initializing
$\wk_{k} = -1$ yields a monotonic decrease of $\ekTrueSO$.

It remains to provide the Jacobian of the fully nonlinear residual
vector~\cref{eqn:nonlinError}.
Define the Loewner-like matrix $\LowSONLk \in \C^{(M - k)\times k}$ entrywise
via
\begin{equation} \label{eqn:LoewnerMat_nonLinJac}
    \left[ \LowSONLk \right]_{i, j} = \frac{\HrkSO(\mu_{i}) - h_{j}}%
      {(\mu_{i} - \lambda_{j}) (\mu_{i} - \sigma_{j})},
\end{equation}
for $i = 1, \ldots, M - k$ and $j = 1, \ldots, k$.
Note that the difference between~\cref{eqn:LoewnerMat_nonLinJac}
and~\cref{eqn:LoewnerMat_SO} is that the value $\FOM(\mu_{i}) = g_{i}$
is replaced by $\HrkSO(\mu_{i})$, the current \AAAFullNL{} approximation
to $g_{i} = G(\mu_{i})$.
Define the vector $\etakDenomWeighted \in \C^{M - k}$ with the entries
\begin{equation}
  \left[ \etakDenomWeighted \right]_{i} = \frac{\eta_{i}}{\denomSOk(\mu_{i})},
\end{equation}
with $\denomSOk(s)$, the denominator rational function in the second-order
barycentric form~\cref{eqn:MechBaryForm} of \AAAFullNL{} approximation at
iteration $k$.
Then, the Jacobian of $\ekTrueSO$ with respect to $\mwk$ is given by
\begin{equation} \label{eqn:JacobianComplexSONLwrtW}
  \JacSONLkW = -\mdiag(\etakDenomWeighted) \LowSONLk,
\end{equation}
and the Jacobian of $\ekTrueSO$ with respect to $\sigkVec$ is entrywise given by
\begin{equation} \label{eqn:JacobianComplexSONLwrtSig}
  \left[ \JacSONLkSig \right]_{i, j} =
    -\frac{w_{j} \left[ \etakDenomWeighted \right]_{i}}%
    {\mu_{i} - \sigma_{j}} \left[ \LowSONLk \right]_{i,j},
\end{equation}
with $i = 1, \ldots, M - k$ and $j = 1, \ldots, k$.
Finally, the Jacobian of the residual vector $\ekTrueSO$ with respect to the
full vector of nonlinear parameters $\begin{bmatrix} \left( \mwk
\right)^{\trans} & \left( \sigkVec \right)^{\trans} \end{bmatrix}^{\trans}
\in \C^{2k}$ is then given as
\begin{equation} \label{eqn:JacobianComplexSONLFull}
  \JacSONLkFull = \begin{bmatrix} \JacSONLkW \\ \JacSONLkSig \end{bmatrix}.
\end{equation}
With the Jacobians at hand, we summarize the fully nonlinear second-order AAA
algorithm (\AAAFullNL) in \Cref{alg:NSOAAA}.
As for the previous second-order AAA methods, we can easily obtain matrices of
the corresponding second-order system realization for \AAAFullNL{}
using~\cref{eqn:MDK_def}.

\begin{algorithm}[t]
  \SetAlgoHangIndent{1pt}
  \DontPrintSemicolon
  \caption{Fully nonlinear second-order AAA algorithm (\AAAFullNL).}%
  \label{alg:NSOAAA}

  \KwIn{Data set $\M = \{(\mu_{i}, g_{i}, \eta_{i})\}_{i = 1}^{M}$ and
    maximum model order $\maxOrder$.}
  \KwOut{Parameters of the second-order barycentric form
    $\lambda_{j}, \sigma_{j}, h_{j}, w_{j}$, for~$j = 1, \ldots, k$.}
  
  \For{$k = 1, \ldots, \maxOrder$}{
    Find $(\mu, g, \eta) \in \M^{(k - 1)}$ that maximizes the weighted
      approximation error
      \begin{equation*}
        (\mu, g, \eta) = \argmax\limits_{(\mu_{i}, g_{i}, \eta_{i})
          \in \M^{(k - 1)}} \eta_{i} \lvert \HrSO^{(k - 1)}(\mu_{i}) - g_{i}
          \rvert.
      \end{equation*}\;
      \vspace{-\baselineskip}

    Update the barycentric parameters
      \begin{equation*}
        \lambda_{k} = \mu, \quad
        h_{k} = g, \quad
        \sigkVec = \begin{bmatrix} \msigma^{(k-1)} \\
          c - \imunit\imag(\lambda_{k}) \end{bmatrix}, \quad
        \mwk = \begin{bmatrix} \mw^{(k - 1)} \\ -1 \end{bmatrix}
      \end{equation*}
      and the data sets
      \begin{equation*}
        \Mk = \M^{(k - 1)} \setminus \{(\mu, g, \eta)\}, \quad
        \PkSO = \PSO^{(k - 1)} \cup \{(\lambda_{k}, h_{k}, \sigk)\}.
      \end{equation*}\;
      \vspace{-\baselineskip}
    
    Solve the nonlinear least-squares problem
      \begin{equation*}
        \min\limits_{\mwk, \sigkVec} \left\lVert \ekTrueSO \right\rVert_{2}^{2}
      \end{equation*}
      for the barycentric weights $\mwk$ and quasi-support points $\sigkVec$.\;
  }
\end{algorithm}


\section{Heuristics on attainable accuracy}%
\label{sec:ModelStructureAndAccuracyHeuristics}

In this section, we first show that a second-order barycentric model can always
be represented as an unstructured barycentric model, but that the converse is
typically not true.
These results provide further justification for enforcing second-order
structure algorithmically as opposed to using a post-processing step, and
leads to heuristics on the attainable accuracy of our second-order AAA
algorithms developed in \Cref{sec:SecondOrderAAAs}.

For simplicity of presentation, we drop here the iteration index $k$ from the
parameters $\sigk$ and $\wk$ but keep the order $k$ of the barycentric models
(i.e., $\Hrk$ and $\HrkSO$) to trim the notation while keeping the order of the
associated barycentric approximations explicit.
We also assume all weights as $\eta_{i} = 1$, though the analysis holds for any
weights.
We remind the reader that an order-$k$ unstructured barycentric model ($\Hrk$)
is a degree $k$ strictly proper rational function, while an order-$k$
second-order barycentric model ($\HrkSO$) is a degree-$2k$ strictly proper
rational function.


\subsection{Switching between model structures}%
\label{sec:SwitchModelStructure}

It is well known~\cite{TisM01} that any degree-$2k$ second-order transfer
function~\cref{eqn:MechTransfer} can be exactly written as an unstructured
degree-$2k$ transfer function~\cref{eqn:UnstructuredTransfer}.
Generally, the converse is not true, that is, an unstructured transfer function
cannot necessarily be written as a second-order transfer function.
In this section, we show that these relations extend to the barycentric
forms~\cref{eqn:UnstructuredBaryForm,eqn:MechBaryForm}.
We begin by showing that any degree-$2k$ stricly proper rational function that
can be represented by the second-order barycentric form~\cref{eqn:MechBaryForm}
can be exactly represented by the unstructured barycentric
form~\cref{eqn:UnstructuredBaryForm}.

\begin{lemma} \label{lmm:2ndOrderIsUnstructured}
  The second-order barycentric form
  \begin{equation} \label{eqn:MechBaryForm2}
    \HrkSO(s) = \frac{\sum_{j = 1}^{k} \frac{h_{j} w_{j}}%
      {(s - \lambda_{j}) (s - \sigma_{j})}}%
      {1 + \sum_{j = 1}^{k} \frac{w_{j}}{(s - \lambda_{j}) (s - \sigma_{j})}}
  \end{equation}
  can be exactly expressed as an unstructured barycentric form
  \begin{equation} \label{eqn:MechBaryAsUnstruct}
    \Hr^{(2k)}(s) = \frac{\sum_{j = 1}^{k} \left( \frac{h_{j} \hat{w}_{j}}%
      {s - \lambda_j} - \frac{h_{j} \hat{w}_{j}}{s - \sigma_{j}} \right)}%
      {1 + \sum_{j = 1}^{k} \left( \frac{\hat{w}_{j}}{s - \lambda_{j}} -
      \frac{\hat{w}_{j}}{s - \sigma_{j}} \right)},
  \end{equation}
  where the unstructured barycentric weights are defined
  \begin{equation}
    \hat{w}_{j} = \frac{w_{j}}{\lambda_{j} - \sigma_{j}}.
  \end{equation}
\end{lemma}
\begin{proof}
  The result follows directly from expanding each term in the numerator and
  denominator of~\cref{eqn:MechBaryForm2} as partial fractions.
\end{proof}

The result of \Cref{lmm:2ndOrderIsUnstructured} is unsurprising, as it has
already been shown in \Cref{lmm:BaryFormRecoverAnything} that any strictly
proper irreducible degree-$2k$ rational function can be represented by a model
of the form~\cref{eqn:UnstructuredBaryForm}.
However, \Cref{lmm:2ndOrderIsUnstructured} yields the explicit relation
between the two barycentric forms in terms of their weights.

We now explore conditions on the parameters of the unstructured barycentric
form~\cref{eqn:UnstructuredBaryForm} that allow to express it equivalently in
the second-order barycentric form.
First, note that the order-$2k$ unstructured barycentric
form~\cref{eqn:MechBaryAsUnstruct} resulting from a second-order barycentric
form has only $k$ unique function value
parameters $h_{j}$, with $j = 1,\ldots, k$.
While this property may seem special, we show in the following that any
degree-$2k$ strictly proper rational function admits a barycentric
representation with only $k$ unique function value parameters.

\begin{lemma} \label{lmm:ValueAttainedTwiceUnique}
  Let $f(s) = \frac{n(s)}{d(s)}$ be a strictly proper, irreducible rational
  function, where the degree of $f(s)$ is $k \geq 2$.
  Then, for all but at most $2k - 2$ values of $h \in \C$, there exist points
  $\lambda_{1} \in \C$ and $\lambda_{2} \in \C$ so that
  $\lambda_{1} \neq \lambda_{2}$ and $f(\lambda_{1}) = f(\lambda_{2}) = h$ hold.
\end{lemma}
\begin{proof}
  Let $h \in \C$ and note that a degree-$k$ strictly proper rational function
  attains each value in $\C$ exactly $k$ times (counting multiplicity).
  If $\lambda_{1} \in \C$ is the only point such that $f(s) = h$, then
  \begin{equation}
    g(s) = f(s) - h = \frac{n(s)}{d(s)} - h
  \end{equation}
  has a zero of multiplicity $k$ at $s = \lambda_{1}$.
  Hence, $f'(\lambda_{1}) = g'(\lambda_{1}) = 0$.
  Since
  \begin{equation}
    f'(s) = \frac{d(s) n'(s) - n(s) d'(s)}{(d(s))^{2}}
  \end{equation}
  has the numerator degree $2k - 2$, the derivative $f'(s)$ has $2k - 2$ zeros,
  counting multiplicity.
  These zeros are shared between $f'(s)$ and $g'(s)$.
  Therefore, there are at most $2k - 2$ values $h$, which may be attained by
  $f$ at only one point in $\C$.
  In particular, those possible values are $f(\lambda_{i})$ where
  $f'(\lambda_{i}) = 0$.
  This concludes the proof.
\end{proof}

\Cref{lmm:ValueAttainedTwiceUnique} shows that a degree-$k$ strictly proper
rational function attains all but at most $2k - 2$ values at least at two
distinct points in $\C$.
To allow easy comparison between arbitrary strictly proper rational functions
in barycentric form and those rational functions, which can be represented with
the second-order barycentric form, a direct consequence of
\Cref{lmm:ValueAttainedTwiceUnique} is presented below.

\begin{corollary} \label{cor:UnstructBaryWithkValues}
  Let $f(s)$ has the form
  \begin{equation}
    f(s) = \frac{\sum_{j = 1}^{2k} \frac{h_{j} w_{j}}{s - \lambda_{j}}}%
      {1 + \sum_{i = j}^{2k} \frac{w_{j}}{s - \lambda_{j}}}.
  \end{equation}
  Then there exist barycentric parameters
  \begin{equation}
    \tilde{\lambda}_{j} \in \C, \quad
    \tilde{\sigma}_{j} \in \C, \quad \text{and} \quad
    \tilde{h}_{j} \in \C, \quad \text{for}\quad j = 1, \ldots, k
  \end{equation}
  and
  \begin{equation}
    \tilde{w}_{j} \in \C, \quad \text{for}\quad j = 1, \ldots, 2k
  \end{equation}
  such that
  \begin{equation} \label{eqn:UnstructInCorrolary}
    f(s) = \frac{\sum_{j = 1}^{k} \left( \frac{\tilde{h}_{j} \tilde{w}_{j}}%
      {s - \tilde{\lambda}_{j}} + \frac{\tilde{h}_{j} \tilde{w}_{j + k}}%
      {s - \tilde{\sigma}_{j}} \right)}%
      {1 + \sum_{j = 1}^{k} \left( \frac{\tilde{w}_{j}}%
      {s - \tilde{\lambda}_{j}} + \frac{\tilde{w}_{j + k}}%
      {s - \tilde{\sigma}_{j}} \right)}
    \end{equation}
    holds.
\end{corollary}
\begin{proof}
  By \Cref{lmm:ValueAttainedTwiceUnique}, for any given degree-$2k$ strictly
  proper rational function we can always choose new support points
  $\tilde{\lambda}_{j}$ and $\tilde{\sigma}_{j}$ such that
  $f(\tilde{\lambda}_{j}) = f(\tilde{\sigma}_{j}) = \tilde{h}_{j}$,
  for $j = 1, \ldots, k$.
  Then, since the unstructured barycentric form can recover any strictly proper
  rational function that does not have poles at its support points
  (\Cref{lmm:BaryFormRecoverAnything}), there must exist
  $\tilde{w}_{j}$, with $j = 1,\ldots, 2k$, such that $f(s)$ can be written
  as~\cref{eqn:UnstructInCorrolary}.
\end{proof}

\Cref{cor:UnstructBaryWithkValues} shows that any degree-$2k$ strictly proper
rational function has a representation in barycentric form with $2k$ distinct
support points but only $k$ distinct function value parameters $h_{j}$,
with $j = 1, \ldots, k$, each appearing exactly twice.
Comparing the expression~\cref{eqn:UnstructInCorrolary}, which can represent
any degree-$2k$ strictly proper  rational function,
with~\cref{eqn:MechBaryAsUnstruct}, which can represent any strictly proper
rational function also representable via the second-order barycentric form, we
observe that the main difference is that~\cref{eqn:UnstructInCorrolary} has
$2k$ (possibly distinct) barycentric weights $\tilde{w}_{j}$,
for $j = 1, \ldots, 2k$, while only $k$ barycentric weights appear
in~\cref{eqn:MechBaryAsUnstruct}.
In the following theorem, we use this observation to show when an arbitrary
strictly proper degree-$2k$ rational function~\cref{eqn:UnstructInCorrolary}
can be represented by the second-order barycentric form.

\begin{theorem} \label{thm:secondOrderRestricted}
  Assume that for a degree-$2k$ strictly proper rational function $\Hr^{(2k)}$
  in unstructured barycentric form, the $2k$ unique support points are chosen
  so that there are only $k$ unique function value parameters $h_{j}$, each
  appearing exactly twice so that
  \begin{equation} \label{eqn:UnstructBary_LamsAndSigs}
    f(s) = \frac{\sum_{j = 1}^{k} \left(\frac{h_{j} w_{j}}{s - \lambda_{j}} +
      \frac{h_{j} w_{j + k}}{s - \sigma_{j}}\right)}%
      {1 + \sum_{j = 1}^{k} \left( \frac{w_{j}}{s - \lambda_{j}} +
      \frac{w_{j + k}}{s - \sigma_{j}}\right)}.
  \end{equation}
  Then, \Cref{eqn:UnstructBary_LamsAndSigs} may be represented in the
  second-order barycentric form~\cref{eqn:MechBaryForm} only if
  \begin{equation}
    w_{j} = -w_{j + k}
  \end{equation}
  holds for the unstructured barycentric weights.
\end{theorem}
\begin{proof}
  Combining the two partial fractions in each term of the sums
  of~\cref{eqn:UnstructBary_LamsAndSigs} gives
  \begin{equation}
    f(s) = \frac{\sum_{j = 1}^{k} h_{j} \frac{s (w_{j} + w_{j + k}) -
      (\sigma_{j} w_{j} + \lambda_{j} w_{j + k})}%
      {(s - \lambda_{j}) (s - \sigma_{j})}}%
      {1 + \sum_{j = 1}^{k} \frac{s(w_{j} + w_{j + k}) -
      (\sigma_{j} w_{j} + \lambda_{j} w_{j + k})}{(s - \lambda_{j})
      (s - \sigma_{j})}}.
  \end{equation}
  To eliminate the linear term in the numerators of the partial fraction sums
  it is necessary that $w_{j} = -w_{j + k}$ as claimed.
\end{proof}

\Cref{thm:secondOrderRestricted} shows that, in general, we \emph{do not}
expect that a degree-$2k$ strictly proper rational function can be represented
by the second-order barycentric form.
In the following section, we use these results to develop heuristics for the
accuracy of our second-order AAA methods.


\subsection{Heuristics on accuracy of second-order AAA methods}%
\label{sec:AccuracyHeuristics}

In this section, we present upper and lower bound heuristics for the error of
our second-order AAA algorithms developed in \Cref{sec:SecondOrderAAAs}.
Specifically, we provide approximate bounds on the error of an order-$k$
approximation obtained by our second-order AAA methods based on the error of
the unstructured (first-order) AAA approximations of degree $k$ and $2k$.
The heuristics we develop in this section are threefold: 
\begin{enumerate}
  \item An order-$2k$ \AAA{} model should have a lower approximation error than
    any order-$k$ second-order AAA model (\Cref{sec:LowerBoundHeuristc}).

  \item An order-$k$ second-order AAA model constructed via \AAAFull{} or
    \AAAFullNL{} should have a lower approximation error than an order-$k$
    \AAA{} model (\Cref{sec:UpperBoundHeuristc}).
  
  \item An order-$k$ \AAALin{} model should have an approximation error nearly
    equal to an order-$k$ \AAA{} model, if the quasi-support points 
    $\sigma_{j}$ are chosen far enough away from the data in $\M$
    (\Cref{sec:UpperBoundHeuristc}).
\end{enumerate}


\subsubsection{A lower bound error heuristic}%
\label{sec:LowerBoundHeuristc}

First, we establish a lower bound heuristic for the error of our second-order
\AAA{} methods at iteration $k$.
Recall from \Cref{sec:SwitchModelStructure} that not every order-$2k$
unstructured barycentric model can be represented as an order-$k$ second-order
barycentric model.
Thus, at iteration $2k$, \AAA{} is forming approximations from the larger
subspace of all degree-$2k$ strictly proper rational functions compared to the
second-order \AAA{} methods at iteration $k$, which construct approximations
from a strict subset of all degree-$2k$ strictly proper rational functions.
Therefore, we expect that an order-$2k$ strictly properrational approximation
produced by \AAA{} should have lower error than a degree-$2k$ strictly proper
rational approximation produced by \AAAFull, \AAALin or \AAAFullNL.

However, we note that this observation remains a heuristic and not a strict
bound due to the linearized least-squares
residual~\cref{eqn:linearizedError_unstruct}.
Since \AAA{} does not minimize the true (nonlinear) least-squares error, we
observe that the second-order \AAA{} methods at iteration $k$ can have smaller
errors than \AAA{} at iteration $2k$.
We observe this most often for \AAAFullNL, since \AAAFullNL{} optimizes the
true nonlinear least-squares residual.

    
\subsubsection{An upper bound error heuristic}%
\label{sec:UpperBoundHeuristc}

Now, we establish a heuristic for the upper bound of the error of our
second-order \AAA{} methods at iteration $k$.
Recall that any order-$k$ second-order barycentric model can be rewritten as
\begin{equation} \label{eqn:MechBaryAsUnstruct_heuristic}
  \Hso(s) = \frac{\sum_{j = 1}^{k} \frac{h_{j} w_{j}}%
    {(s - \lambda_{j}) (s - \sigma_{j})}}%
    {1 + \sum_{j = 1}^{k} \frac{w_{j}}{(s - \lambda_{j}) (s - \sigma_{j})}} 
  = \frac{\sum_{j = 1}^{k} \frac{w_{j}}{\lambda_{j} - \sigma_{j}}
    \left( \frac{h_{j}}{s - \lambda_{j}} - \frac{h_{j}}{s - \sigma_{j}}
    \right)}{1 + \sum_{j = 1}^{k} \frac{w_{j}}{\lambda_{j} -\sigma_{j}} 
    \left(\frac{1}{s - \lambda_{j}} - \frac{1}{s - \sigma_{j}}\right)};
\end{equation}
cf. \Cref{lmm:2ndOrderIsUnstructured}.
This equivalence leads to a modified method to recover the barycentric weights
of a \AAALin{} model.
Following similar residual calculations as in \Cref{sec:AAA_Unstructured}, we
define the Loewner-like matrix $\LowUStoSO \in \C^{(M - k) \times k}$ as
\begin{equation} \label{eqn:LoewnerMat_US2SO}
  \left[ \LowUStoSO \right]_{i, j} = \frac{g_{i} - h_{j}}%
    {\mu_{i} - \lambda_{j}} - \frac{g_{i} - h_{j}}{\mu_{i} - \sigma_{j}},
\end{equation}
with $(\mu_{i}, g_{i}, \eta_{i}) \in \M$ and
$(\lambda_{j}, h_{j}, \sigma_{j}) \in \PSO$,
for $i = 1, \ldots, M - k$ and $j = 1, \ldots, k$.
Then, the coefficients $w_{1}, \ldots, w_{k}$ that minimize the separable
residual vector~\cref{eqn:separableObjFun} for fixed
$\sigma_{1}, \ldots, \sigma_{k}$ may be recovered by either
computing~\cref{eqn:LinearLSSolution_SOLAAA} with the Loewner-like matrix
$\LowSO$ defined in~\cref{eqn:LoewnerMat_SO}, or by computing
\begin{equation} \label{eqn:LS_US2SO}
  \hat{\mw} = -\left( \LowUStoSO \right)^{\dagger} \mg
\end{equation}
and recovering the barycentric weights for the second-order model via
\begin{equation}
  w_{j} = (\lambda_{j} - \sigma_{j}) \hat{w}_{j}.
\end{equation}
This reveals the connection between \AAALin{} and \AAA.

\begin{theorem} \label{thm:SOConvergeToFO}
  Let $\LowUS$ be defined as in~\cref{eqn:LoewnerMat_USAAA} for the parameters
  $\lambda_{j}, h_{j}$, with $j = 1$, $\ldots, k$, and the data
  $\mu_{i}, g_{i}$, with $i = 1, \ldots, M - k$.
  For the same parameters and data let $\LowUStoSO$ be defined as
  in~\cref{eqn:LoewnerMat_US2SO} using the additional quasi-support points
  $\sigma_{1}, \ldots, \sigma_{k}$.
  Let
  \begin{equation}
    \Hr = \frac{\sum_{j = 1}^{k}\frac{h_{j} w_{j}}{s - \lambda_{j}}}%
      {1 + \sum_{j = 1}^{k} \frac{w_{j}}{s - \lambda_{j}}}
  \end{equation}
  be the unstructured barycentric model obtained by
  solving~\cref{eqn:WeightedLS_USAAA} with $\LowUS$ for the barycentric weights
  $w_{1}, \ldots, w_{k}$.
  Further, let
  \begin{equation}
    \HrSO = \frac{\sum_{j = 1}^{k} \frac{h_{j} \hat{w}_j
      (\lambda_{j} - \sigma_{j})}{(s - \lambda_{j}) (s - \sigma_{j})}}%
      {1 + \sum_{j = 1}^{k} \frac{\hat{w}_{j} (\lambda_{j} - \sigma_{j})}%
      {(s - \lambda_{j}) (s - \sigma_{j})}}
  \end{equation}
  be the second-order barycentric model obtained by solving~\cref{eqn:LS_US2SO}
  with $\LowUStoSO$ for the linear parameters
  $\hat{w}_{1}, \ldots, \hat{w}_{k}$.
  Then, if $\LowUS^{\dagger}$ has constant rank in some open neighborhood, we
  have that
  \begin{equation}
    \lim_{\min\limits_{j} (\lvert \real(\sigma_{j}) \rvert) \to \infty}
      \lVert \Hr - \HrSO \rVert_{\Linf} = 0,
  \end{equation}
  where $\rVert . \lVert_{\Linf}$ denotes the $\Linf$-norm.
\end{theorem}
\begin{proof}
  First, observe that if all $\lvert \real(\sigma_{j}) \rvert \to \infty$,
  then $\LowUStoSO \to \LowUS$.
  Since $(\LowUS)^{\dagger}$ has constant rank in some open neighborhood and
  the pseudoinverse is a continuous function whenever the matrix has constant
  rank~\cite{Ste69}, we also have $\LowUStoSO^{\dagger} \to \LowUS^{\dagger}$.
  Thus, as $\min(\lvert \real(\sigma_{j}) \rvert) \to \infty$, the second-order
  barycentric weights recovered via~\cref{eqn:LS_US2SO} converge to the
  unstructured barycentric weights recovered via~\cref{eqn:WeightedLS_USAAA},
  that is, $\hat{w}_{j} \to {w}_{j}$.
  Evaluating the second-order barycentric form at any $\imunit\omega$ for
  $\omega \in \R$ where each $\lvert \real(\sigma_j) \rvert$ is large, we have
  \begin{equation}
    \HrSO(\imunit \omega) = \frac{\sum_{j = 1}^{k} \frac{h_{j} \hat{w}_{j}
      (\lambda_{j} - \sigma_{j})}{(\imunit \omega - \lambda_{j})%
      (\imunit \omega - \sigma_{j})}}{1 + \sum_{j = 1}^{k}
      \frac{\hat{w}_{j} (\lambda_{j} - \sigma_{j})}%
      {(\imunit \omega - \lambda_{j}) (\imunit \omega - \sigma_{j})}}
    = \frac{\sum_{j = 1}^{k} \frac{h_{j} \hat{w}_{j}}%
      {(\imunit \omega - \lambda_{j})}(1 + \epsilon_{j})}%
      {1 + \sum_{j = 1}^{k} \frac{\hat{w}_{j}}{(\imunit \omega - \lambda_{j})}
      (1 + \epsilon_{j})},
  \end{equation}
  where $\epsilon_{j} \in \C$ can be made arbitrarily small by increasing
  $\lvert \real(\sigma_{j}) \rvert$.
  Therefore, for any $\imunit \omega$, we have that
  \begin{equation}
    \lim_{\min\limits_{j}(\lvert \real(\sigma_{j}) \rvert) \to \infty}
      (\Hr(\imunit \omega) - \HrSO(\imunit \omega)) \to 0
  \end{equation}
  holds, that is, $\HrSO$ converges to $\Hr$ uniformly on the imaginary axis
  and, therefore,
  \begin{equation}
    \lim_{\min\limits_{j}(\lvert \real(\sigma_{j}) \rvert) \to \infty}
      \lVert \Hr - \HrSO \rVert_{\Linf} \to 0,
  \end{equation}
  which proves the result.
\end{proof}

The implication of \Cref{thm:SOConvergeToFO} is that if the quasi-support points
$\sigma_{j}$ in the second-order barycentric formare taken to be far in the
left half plane, we will approximately recover the unstructured barycentric
model formed with the same $\lambda_{j}$ and $h_{j}$ parameters.
This provides an upper bound heuristic on the error in our second-order AAA
models.
An order-$k$ second-order AAA approximation to a data set $\M$ attained by
either \AAAFull{} or \AAAFullNL{} should be able to perform as well or better
than the unstructured order-$k$ \AAA{} approximation to the same data.
Our results in~\Cref{sec:numerics} show that both \AAAFull{} and \AAAFullNL{}
typically greatly outperform the unstructured \AAA{} algorithm for the same
model order~$k$.
Since in \AAALin, the quasi-support points $\sigma_{j}$ are only initialized
(and not moved) far in the left-half plane, \Cref{thm:SOConvergeToFO} also
suggests that \AAALin{} should have accuracy comparable to \AAA.
Indeed, we empirically observe this behavior in our numerical experiments in
\Cref{sec:numerics}.


\section{Constructing real-valued models}%
\label{sec:realification}

In applications that involve the time domain simulation of dynamical systems,
it is typically required that the data-driven model has a real state-space
representation, i.e., all matrices in~\cref{eqn:Lam_w_h_def}
or~\cref{eqn:MDK_def} have real entries.
We can enforce that models constructed by any of the algorithms described in
\Cref{sec:AAA_Unstructured,sec:SecondOrderAAAs} admit a real state-space
realization by requiring that the sets of parameters $\lambda_{j}$, $h_{j}$, and
$w_{j}$ (and $\sigma_{j}$ for the second-order methods) are all closed under
complex conjugation.
For the second-order methods, we also require the support points and quasi
support points to have the same conjugacy pattern, that is, if
$\lambda_{j} = \overline{\lambda}_{j + 1}$, then we also have
$\sigma_{j} = \overline{\sigma}_{j + 1}$.
Under these conditions, there exist state-space transformations that
allow to transform all matrices in~\cref{eqn:Lam_w_h_def,eqn:MDK_def}
into real form.
The remainder of this section provides the necessary modifications to the 
algorithms presented in \Cref{sec:AAA_Unstructured,sec:SecondOrderAAAs} to
ensure that real models are recovered.

For simplicity of presentation, in what follows we drop the iteration index~$k$
and assume that our data set $\M$ contains only sample points with positive
imaginary part, that is, if $(\mu_{i}, g_{i}, \eta_{i}) \in \M$, then
$\imag(\mu_{i}) > 0$.
Additionally, we define here the data and parameter sets needed in the
upcoming sections.
Recall that the set $\Pset$ contains the barycentric parameters of $\Hr$ except
the barycentric weights.
We now define the set $\PConj$ to contain the conjugates of each element in
$\Pset$ so that
\begin{equation}
  \PConj = \{(\overline{\lambda}_{1}, \overline{h}_{1}), \ldots ,
    (\overline{\lambda}_{k}, \overline{h}_{k})\}.
\end{equation}
Similarly, for the second-order methods we define the set
\begin{equation}
  \PSOConj = \{
    (\overline{\lambda}_{1}, \overline{h}_{1}, \overline{\sigma}_{1}), \ldots,
    (\overline{\lambda}_{k}, \overline{h}_{k}, \overline{\sigma}_{k}) \}.
\end{equation}
Additionally, we define the augmented parameter sets
\begin{equation}
  \PAll = \Pset \cup \PConj
  \quad\text{and}\quad
  \PSOAll = \PSO \cup \PSOConj.
\end{equation}
Finally, for the data set $\M = \{(\mu_{i}, g_{i}, \eta_{i})\}_{i = 1}^{M}$, we
also define the conjugate and augmented data sets $\MConj$ and $\MAll$ via
\begin{equation}
  \MConj = \{ (\overline{\mu}_{i}, \overline{g}_{i}, \eta_{i})\}_{i = 1}^{M}
  \quad\text{and}\quad
  \MAll = \M \cup \MConj.
\end{equation}


\subsection{Realification of the unstructured AAA}%
\label{sec:RealUSAAA}

We assume to be at iteration $k$ of \AAA.
Then, the set
\begin{equation}
  \Pset = \{(\lambda_{1}, h_{1}), \ldots, (\lambda_{k}, h_{k})\},
  \quad\text{with each}\quad \imag(\lambda_{j}) > 0,
\end{equation}
was constructed to approximate the data in $\M$.
To recover a system~\cref{eqn:UnstructuredSystem} with real state-space
matrices that interpolates at the chosen points
$\lambda_{1}, \ldots, \lambda_{k}$, we have to use the barycentric parameters
\begin{equation}
  \PAll = \{ (\lambda_{1}, h_{1}), \ldots, (\lambda_{k}, h_{k}),
    (\overline{\lambda}_{1}, \overline h_{1}), \ldots,
    (\overline{\lambda}_{k}, \overline h_{k})\},
\end{equation}
which approximate the data in $\MAll$.
In exact arithmetic, if the Loewner matrix~\cref{eqn:LoewnerMat_USAAA}, and the
function value and weight vectors~\cref{eqn:GkAndEtakDef} are formed with the
data in $\MAll$ and the parameters in $\PAll$, then the vector of weights
recovered via~\eqref{eqn:WeightedLS_USAAA} will be closed under conjugation.
Therefore, we can write
\begin{equation}
  \mw = \begin{bmatrix} w_{1} & \overline{w}_{1} & \ldots & w_{k} &
    \overline{w}_{k} \end{bmatrix}^{\trans}.
\end{equation}
However, due to the finite-precision arithmetic, the recovered vector of
barycentric weights will not be closed under conjugation, leading to dynamical
systems that cannot be represented via a real state-space realization.
To resolve this issue, instead of solving~\cref{eqn:WeightedLS_USAAA} with the
augmented data matrix $\MAll$ and augmented parameter set $\PAll$, we  exploit
the linearity of the optimization problem and solve directly for the real and
imaginary parts of the barycentric weights $w_{j}$.
In particular, define the vector of non-redundant real and imaginary parts of
$\mw$ as
\begin{equation} \label{eqn:RealWVec_def}
  \mwReal \coloneqq \begin{bmatrix} \real(w_{1}) & \imag(w_{1}) & \ldots &
    \real(w_{k}) & \imag(w_{k}) \end{bmatrix}^{\trans} \in \R^{2k}.
\end{equation}
For $(\mu_{i}, g_{i}, \eta_{i}) \in \M$, define the vector of real and
imaginary parts of the function values $\mgReal \in \R^{2(M - k)}$ and the
augmented vector of data weights $\ETAReal \in \R^{2(M-k)}$ as 
\begin{subequations} \label{eqn:RealHVec_def}
\begin{align}
  \mgReal & = \begin{bmatrix} \real(g_{1}) & \imag(g_{1}) & \ldots &
    \real(g_{M - k}) & \imag(g_{M - k}) \end{bmatrix}, \\
  \ETAReal & = \begin{bmatrix} \eta_{1} & \eta_{1} & \eta_{2} & \eta_{2} &
    \ldots & \eta_{M - k} & \eta_{M - k} \end{bmatrix}
\end{align}
\end{subequations}
Finally, for $(\mu_{i}, g_{i}, \eta_{i}) \in \M$ and
$(\lambda_{j}, h_{j}) \in \Pset$, define the coefficients
\begin{equation} \label{eqn:termdef_AAAUnstruct}
  \alpha_{i, j} = \frac{g_{i} - h_{j}}{\mu_{i} - \lambda_{j}}
  \quad\text{and}\quad
  \beta_{i, j} = \frac{g_{i} - \overline{h}_{j}}{\mu_{i} -
    \overline{\lambda}_{j}},
\end{equation}
for $i = 1, \ldots, M - k$ and $j = 1, \ldots, k$.
Then, the real matrix $\LowUSReal \in \R^{2M\times 2k}$ defined entrywise
by
\begin{subequations} \label{eq:L_AAAUnstruct_def}
\begin{align}
  \left[ \LowUSReal \right]_{2i - 1, 2j - 1} & =
    -\real(\alpha_{i,j} + \beta_{i,j}), &
  \left[ \LowUSReal \right]_{2i - 1, 2j} & =
    \imag(\alpha_{i, j} - \beta_{i, j}), \\
  \left[ \LowUSReal \right]_{2i, 2j - 1} & =
    -\imag(\alpha_{i, j} + \beta_{i, j}), &
  \left[ \LowUSReal \right]_{2i - 1, 2j - 1} & =
    -\real(\alpha_{i, j} - \beta_{i, j}),
\end{align}
\end{subequations}
for $i = 1, \ldots, M - k$ and $j = 1, \ldots, k$, allows us to solve for the
real and imaginary parts of the barycentric weights via
\begin{equation} \label{eqn:WeightedLS_USAAA_real}
  \mwReal = \left( \mdiag(\ETAReal) \LowUSReal \right)^{\dagger}
    \left( \mdiag(\ETAReal) \mgReal \right).
\end{equation}
To recover a corresponding state-space realization with real matrices, let
\begin{equation}
  \widetilde{\mA}_{j} = \begin{bmatrix} \phantom{-}\real(\lambda_{j}) &
    \imag(\lambda_{j}) \\ -\imag(\lambda_{j}) & \real(\lambda_{j})
    \end{bmatrix}, \quad\text{for}\quad j = 1, \ldots, k,
\end{equation}
and define $\widetilde{\mA} \in \R^{2k\times2k}$ to be the block diagonal
matrix of the form
\begin{equation}
  \widetilde{\mA} = \mdiag(\widetilde{\mA}_{1}, \ldots, \widetilde{\mA}_{k}).
\end{equation}
Further, we define the vectors
\begin{subequations}
\begin{align}
  \widetilde{\mc} & = \begin{bmatrix} \real(h_{1}) & \imag(h_{1}) & \ldots &
    \real(h_{k}) & \imag(h_{k}) \end{bmatrix}^{\trans}, \\
  \widetilde{\mb} & = \begin{bmatrix} \real(w_{1}) & -\imag(w_{1}) & \ldots &
    \real(w_{k}) & -\imag(w_{k}) \end{bmatrix}^{\trans} \quad\text{and}\\
  \widetilde{\mz} & = \begin{bmatrix} 2 & 0 & 2 & 0 & \ldots & 2 & 0
    \end{bmatrix}^{\trans}.
\end{align}
\end{subequations}
Then, a real state-space realization for the parameters in $\PAll$ and
$\mwReal$ is given by
\begin{equation}
  \mE = \eye{2k}, \quad
  \mA = \widetilde{\mA} - \widetilde{\mb} \widetilde{\mz}^{\trans}, \quad
  \mb = \sqrt{2}\,\widetilde{\mb}, \quad
  \mc = \sqrt{2}\,\widetilde{\mc}.
\end{equation}
A related approach has been proposed in~\cite{Hoc17} for the preservation of
Hermitian symmetry structure in the classical AAA algorithm~\cite{NakST18}.


\subsection{Realification of second-order AAAs}%
\label{sec:RealMechAAA}

In this section, we discuss the modifications necessary to the second-order AAA
methods to construct models with real state-space realizations.
As in the unstructured case, we assume to be at iteration $k$ so that
\begin{equation}
  \PSO = \{ (\lambda_{1}, h_{1}, \sigma_{1}), \ldots,
    (\lambda_{k}, h_{k}, \sigma_{k})\},
    \quad\text{with each}\quad \imag(\lambda_{j}) > 0,
\end{equation}
were chosen to approximate the data in $\M$.
In the following discussions, we will require matrices with the same structure
of $\LowSO$ and $\LowSONL$
in~\cref{eqn:LoewnerMat_SO,eqn:LoewnerMat_nonLinJac}, but where the
data used to construct them comes from the sets with conjugate data.
To make this dependence explicit, we use the notation
\begin{equation}
  \LowSO = \LowSO(\PSO, \M)
\end{equation}
to show that the matirx $\LowSO$ as defined in~\cref{eqn:LoewnerMat_SO} is
defined from the parameters in $\PSO$ and data in $\M$.
Then, $\LowSO(\PSOAll, \M)$ is defined as in~\cref{eqn:LoewnerMat_SO} but with
the parameters in $\PSOAll$ instead of $\PSO$.
Similar notation is used for $\LowSONL$.


\subsubsection{Realification of the separable second-order AAA}%
\label{sec:real_SOAAA}

First, we consider \AAAFull, the second-order AAA method with separable
nonlinear least-squares objective function.
To recover second-order systems that admit real state-space realizations, we
define the residual vector $\eSO \in \C^{2(M - k)}$ for the parameters in
$\PSO$ and data in $\M$ entrywise to be
\begin{subequations} \label{eqn:separableError_real}
\begin{align}
  \left[ \eSO \right]_{i} & = \eta_{i} \left( g_{i} - \sum\limits_{j = 1}^{k}
    \left( w_{j} \frac{h_{j} - g_{i}}%
    {(\mu_{i} - \lambda_{j}) (\mu_{i} -\sigma_{j})} +
    \overline{w}_{j} \frac{\overline{h}_{j} - g_{i}}%
    {(\mu_{i} - \overline{\lambda}_{j}) (\mu_{i} - \overline{\sigma}_{j})}
    \right) \right), \\
  \left[ \eSO \right]_{i + M - k} & = \eta_{i} \left( \overline{g}_{i} -
    \sum\limits_{j = 1}^{k} \left( w_{j} \frac{h_{j} - \overline{g}_{i}}%
    {(\overline{\mu}_{i} - \lambda_{j}) (\overline{\mu}_{i} - \sigma_{j})} +
    \overline{w}_{j} \frac{\overline{h}_{j} - \overline{g}_{i}}%
    {(\overline{\mu}_{i} - \overline{\lambda}_{j})
    (\overline{\mu}_{i} - \overline{\sigma}_{j})} \right) \right),
\end{align}
\end{subequations}
for $i = 1, \ldots, M - k$.
For each $j = 1,\ldots, k$, the residual~\cref{eqn:separableError_real} is
differentiable with respect to $\sigma_{j}$ when every $\overline{\sigma}_{j}$
is regarded as a constant, and is differentiable with respect to
$\overline{\sigma}_{j}$ when every $\sigma_{j}$ is regarded as constant.
Thus, the Wirtinger calculus may be used to
minimize the residual vector~\cref{eqn:separableError_real}.
Recall that $\LowSO(\PSO, \M)$ is the Loewner-like
matrix~\cref{eqn:Jacobian_SO} where the dependence on the parameters $\PSO$ and
$\M$ is explicit.
Then, the Jacobian of the residual vector~\cref{eqn:separableError_real} with
respect to $\msigma$ is given by
\begin{equation} \label{eqn:JacobianRealVarPro_notConj}
  \left[ \JacSONoConj \right]_{i, j} = \left[ \LowSO
    \left( \PSO, \MAll \right) \right]_{i, j}
    \frac{\eta_{i}}{\mu_{i} - \sigma_{j}},
\end{equation}
and the Jacobian of the residual vector with respect to $\msigmaConj$ is
\begin{equation} \label{eqn:JacobianRealVarPro_Conj}
    \left[ \JacSOConj \right]_{i, j} = \left[
      \LowSO(\PSOConj, \MAll) \right]_{i, j}
      \frac{\eta_{i}}{\mu_{i} - \overline{\sigma}_{j}},
\end{equation}
for $i = 1, \ldots, 2(M - k)$ and $j = 1, \ldots, k$.

As in \Cref{sec:SOAAA}, \VP{} only explicitly optimizes the nonlinear
parameters $\msigma$ and $\msigmaConj$ so that the barycentric weights
$\mw$ must be recovered after the optimization step.
As done in \Cref{sec:RealUSAAA}, we aim to solve for~\cref{eqn:RealWVec_def},
the real and imaginary parts of the barycentric weights.
Recall the definitions of $\mgReal$ and $\ETAReal$ in~\cref{eqn:RealHVec_def}.
For the data $(\mu_{i}, g_{i}, \eta_{i}) \in \M$ and the barycentric parameters
$(\lambda_{j}, h_{j}, \sigma_{j}) \in \PSO$, define
\begin{equation} \label{eqn:termdef_SOAAA}
  \alpha_{i, j} = \frac{g_{i} - h_{j}}{(\mu_{i} - \lambda_{j})
    (\mu_{i} - \sigma_{j})}
  \quad\text{and}\quad
  \beta_{i, j} = \frac{g_{i} - \overline{h}_{j}}%
    {(\mu_{i} - \overline{\lambda}_{j})(\mu_{i} - \overline{\sigma}_{j})},
\end{equation}
for $i = 1, \ldots, M - k$ and $j = 1, \ldots, k$.
Then, the real matrix $\LowSOReal \in \R^{2M \times 2k}$ defined by 
\begin{subequations} \label{eqn:L_AAALin_def}
\begin{align}
  \left[ \LowSOReal \right]_{2i - 1,2j - 1} & =
    -\real(\alpha_{i, j} + \beta_{i, j}), &
  \left[ \LowSOReal \right]_{2i - 1, 2j} & =
    \imag(\alpha_{i, j} - \beta_{i, j}), \\
  \left[ \LowSOReal \right]_{2i, 2j - 1} & =
    -\imag(\alpha_{i, j} + \beta_{i, j}), &
  \left[ \LowSOReal \right]_{2i - 1, 2j - 1} & =
    -\real(\alpha_{i,j} - \beta_{i,j}),
\end{align}
\end{subequations}
for $i = 1, \ldots, M - k$ and $j = 1, \ldots, k$, allows us to solve for the
real and imaginary parts of the barycentric weights via
\begin{equation} \label{eqn:SolveForW_AAALIN_real}
  \mwReal = \left( \mdiag\left( \ETAReal \right) \LowSOReal
    \right)^{\dagger} \left( \mdiag\left( \ETAReal \right) \mgReal\right).
\end{equation}

To recover a real second-order state-space realization, we define the
matrices
\begin{equation}
  \widetilde{\mD}_{j} = \begin{bmatrix} -\real(\lambda_{j} + \sigma_{j}) &
    -\imag(\lambda_{j} + \sigma_{j}) \\ \phantom{-}\imag(\lambda_{j} + \sigma_{j}) &
    -\real(\lambda_{j} + \sigma_{j}) \end{bmatrix},
  \quad\text{for}\quad j = 1, \ldots, k,
\end{equation}
and
\begin{equation}
  \widetilde{\mK}_{j} = \begin{bmatrix} \phantom{-}\real(\lambda_{j} \sigma_{j}) &
    \imag(\lambda_{j} \sigma_{j}) \\ -\imag(\lambda_{j} \sigma_{j}) &
    \real(\lambda_{j} \sigma_{j}) \end{bmatrix},
    \quad\text{for}\quad j = 1, \ldots, k.
\end{equation}
Next, we define $\widetilde{\mD} \in \R^{2k\times2k}$ and
$\widetilde{\mK} \in \R^{2k\times2k}$ to be the block diagonal matrices
of the form
\begin{equation}
  \widetilde{\mD} = \mdiag(\widetilde{\mD}_{1}, \ldots, \widetilde{\mD}_{k})
  \quad\text{and}\quad
  \widetilde{\mK} = \mdiag(\widetilde{\mK}_{1}, \ldots, \widetilde{\mK}_{k}).
\end{equation}
Finally, define the vectors
\begin{subequations}
\begin{align}
  \widetilde{\mc} & = \begin{bmatrix} \real(h_{1}) & \imag(h_{1}) & \ldots &
    \real(h_{k}) & \imag(h_{k}) \end{bmatrix}^{\trans},\\
  \widetilde{\mb} & = \begin{bmatrix} \real(w_{1}) & -\imag(w_{1}) & \ldots &
    \real(w_{k}) & -\imag(w_{k}) \end{bmatrix}^{\trans} \quad\text{and} \\
  \widetilde{\mz} & = \begin{bmatrix} 2 & 0 & 2 & 0 & \ldots & 2 & 0
    \end{bmatrix}^{\trans}.
\end{align}
\end{subequations}
Then, a real state-space realization for the parameters in $\PAll$ and
$\mwReal$ is given by
\begin{equation} \label{eqn:realRealization_SO}
  \mM = \eye{2k}, \quad
  \mD = \widetilde{\mD}, \quad
  \mK = \widetilde{\mK} + \widetilde{\mb}\widetilde{\mz}^{\trans}, \quad
  \mb = \sqrt{2}\,\widetilde{\mb}, \quad
  \mc = \sqrt{2}\,\widetilde{\mc}.
\end{equation}


\paragraph{Realification of the linearized second-order AAA}
As \AAALin{} does not optimize the quasi-support points $\sigma_{j}$ beyond
their initialization and only solves a linear least-squares problem for the
barycentric weights $w_{j}$, the only modification to \AAALin{} to recover
systems with real state-space realizations is to solve for the real and
imaginary parts of the barycentric weights
via~\cref{eqn:SolveForW_AAALIN_real}.
Then, the real state-space matrices are recovered
using~\cref{eqn:realRealization_SO}.


\subsubsection{Realification of the fully nonlinear second-order AAA}%
\label{sec:real_SONLAAA}

Similar to constructing real models with \AAAFull, we need to only supply
gradients of the fully nonlinear residual vector~\cref{eqn:nonlinError} when
constructed with the parameters in $\PSOAll$ and the augmented weight
barycentric vector $\begin{bmatrix} \mw & \overline{\mw} \end{bmatrix}$.
The $i$-th entry of this residual vector is given as
\begin{equation} \label{eqn:NLError_real}
  \left[ \eTrueSO \right]_{i} = \eta_{i} \left\lvert
    \frac{\sum_{j = 1}^{k} \left( \frac{h_{j} w_{j}}{(\mu_{i} - \lambda_{j})
    (\mu_{i} - \sigma_{j})} + \frac{\overline{h}_{j} \overline{w}_{j}}%
    {(\mu_{i} -\overline{\lambda}_{j}) (\mu_{i} - \overline{\sigma}_{j})}
    \right)}{1 + \sum_{j = 1}^{k} \left( \frac{w_{j}}{(\mu_{i} - \lambda_{j})
    (\mu_{i} - \sigma_{j})} + \frac{\overline{w}_{j}}%
    {(\mu_{i} - \overline{\lambda}_{j}) (\mu_{i} - \overline{\sigma}_{j})}
    \right)} - g_{i} \right\rvert.
\end{equation}
Recall that $\LowSONL(\PSO, \M)$ is the Loewner-like
matrix~\cref{eqn:LoewnerMat_nonLinJac} where the dependence on $\PSO$ and $\M$
is explicit.
Then, the Jacobian of the residual vector~\cref{eqn:NLError_real} with respect
to $\mw$ is given by
\begin{equation} \label{eqn:JacobianRealSONLwrtW}
  \JacRealSONLkW = \mdiag(\etaDenomWeighted) \LowSONL(\PSO,\M)
    \in \C^{2(M-k) \times k},
\end{equation}
and the Jacobian of the residual vector~\cref{eqn:NLError_real} with respect to
$\msigma$ is defined entrywise as
\begin{equation} \label{eqn:JacobianRealSONLwrtSig}
  \left[ \JacRealSONLkSig \right]_{i, j} = \frac{w_{j} \left[ \etaDenomWeighted
    \right]_{i}}{\mu_{i} - \sigma_{j}} \left[ \LowSONL(\PSO, \M) \right]_{i, j},
\end{equation}
for $i = 1, \ldots 2(M - k)$ and $j = 1, \ldots, k$.
Then, as in \Cref{sec:SONLAAA}, the Jacobian with respect to the complete
non-conjugated parameter vector is
\begin{equation} \label{eqn:JacobianRealSONLUnconj}
  \JacRealSONLkFull = \begin{bmatrix} \JacRealSONLkW \\
    \JacRealSONLkSig \end{bmatrix}.
\end{equation}
Similarly, the Jacobians of the fully nonlinear
residual~\cref{eqn:NLError_real} with respect to $\mwConj$ and $\msigmaConj$
are given by
\begin{equation} \label{eqn:JacobianRealSONLwrtWConj}
  \JacRealSONLkWConj = \mdiag(\etaDenomWeighted)
    \left[ \LowSONL(\PSOConj,\M) \right] \in \C^{2(M-k) \times k}
\end{equation}
and
\begin{equation} \label{eqn:JacobianRealSONLwrtSigConj}
  \left[ \JacRealSONLkSigConj \right]_{i, j} =
    \frac{\overline{w}_{j} \left[ \etaDenomWeighted \right]_{i}}%
    {\mu_{i} - \overline{\sigma}_{j}} \left[ \LowSONL(\PSOConj, \M)
    \right]_{i, j},
\end{equation}
for $i = 1, \ldots 2(M - k)$ and $j = 1, \ldots, k$, respectively.
Finally, the Jacobian with respect to the full conjugated parameter vector is
\begin{equation} \label{eqn:JacobianRealSONLConj}
  \JacRealSONLkFullConj = \begin{bmatrix} \JacRealSONLkWConj \\
    \JacRealSONLkSigConj \end{bmatrix}.
\end{equation}

The two Jacobians~\cref{eqn:JacobianRealSONLUnconj,eqn:JacobianRealSONLConj}
allow us to minimize the residual~\cref{eqn:NLError_real} with respect to the
parameters $\msigma$ and $\mw$ using the Wirtinger calculus.
After optimization, the real state-space matrices may be recovered
via~\cref{eqn:realRealization_SO}.


\section{Numerical experiments}%
\label{sec:numerics}

In this section, we demonstrate the effectiveness of our second-order AAA
algorithms developed in \Cref{sec:SecondOrderAAAs} in three numerical examples.
The experiments reported here were performed on a 2023 MacBook Pro equipped
with 16\,GB RAM and an Apple M2 Pro chip.
Computations were done in MATLAB 24.1.0.2578822 (R2024a) running on
macOS Sequoia 15.4.
For numerical optimization involving the Wirtinger calculus, we use the
TensorLab toolbox provided by KU Leuven~\cite{VerDSetal16}.
A description of the methods can be found in~\cite{SorVD12}.
The source code, data and results of the numerical experiments are open
source/open access and available at~\cite{supAckW25}.


\subsection{Experimental setup}%
\label{sec:setup}

We compare the performance of our second-order AAA algorithms against the
standard unstructured AAA algorithm (\AAA).  
Motivated by the discussion in \Cref{sec:AccuracyHeuristics}, we also compare
the proposed methods to \AAA{} with twice the reduced order, i.e., we run
\Cref{alg:UnstructAAA} to order $2\maxOrder$ instead of $\maxOrder$.
We will refer to the models computed in this manor by \AAAtwo.
Note that \AAAtwo{} is exactly the same as \AAA{} except that at iteration $k$,
\AAAtwo{} constructs a degree-$2k$ rational function instead of degree $k$.

For the comparison of the constructed data-driven models in terms of
approximation accuracy, we define the following error measures.
Given an order-$k$ barycentric model $\Hr$ of either the
form~\cref{eqn:UnstructuredBaryForm} or~\cref{eqn:MechBaryForm} approximating
the data of an unknown function $H$ given in $\M$, we define the weighted
error vector $\ekResults$ entrywise as
\begin{equation} \label{eqn:error_weighted_results}
  \left[ \ekResults \right]_{i} =
    \eta_{i} \lvert \Hr(\mu_{i}) - H(\mu_{i}) \rvert =
    \eta_{i} \lvert \Hr(\mu_{i}) - g_{i} \rvert,
\end{equation}
with $(\mu_{i}, g_{i}, \eta_{i}) \in \M$, for $i = 1, \ldots, M$.
Then, with the weighted data vector
\begin{equation}
  \gWeighted = \begin{bmatrix} \eta_{1} \lvert g_{1} \rvert &
    \eta_{2} \lvert g_{2} \rvert & \ldots &
    \eta_{M} \lvert g_{M} \rvert \end{bmatrix}^{\trans},
\end{equation}
we define the weighted relative $\Ltwo$ approximation error as
\begin{equation} \label{eqn:ltwoerror}
  \LtwoWeighted = \frac{\lVert \ekResults \rVert_{2}}%
    {\lVert \gWeighted \rVert_{2}},
\end{equation}
the weighted relative $\Linf$ approximation error as
\begin{equation} \label{eqn:linferror}
  \LinfWeighted = \frac{\lVert \ekResults \rVert_{\infty}}%
    {\lVert \gWeighted \rVert_{\infty}},
\end{equation}
and the weighted maximum pointwise relative approximation error as
\begin{equation} \label{eqn:ptwerror}
  \relerr = \max\limits_{i = 1, \ldots, M} \left(
    \frac{\left[ \ekResults \right]_{i}}{\left[ \gWeighted \right]_{i}} \right).
\end{equation}

To effectively evaluate the performance of the data-driven models over varying
orders, we will make use of the MORscore~\cite{Him21, Wer21}.
In short, the MORscore is a measure that compresses the information in a
relative error-per-order plot into a single, easily comparable number.
Given a relative error plot $(k, \epsilon(k)) \in \N \times [0, 1]$, where $k$
is the order of the model and $\epsilon(k)$ the relative error in a
user-defined measure, and a minimum attainable error $\epsilon_{\min}$, the
MORscore is defined as the area below the normalized relative error plot
$(\phi_{k}, \phi_{\epsilon(k)})$, where
\begin{equation}
  \phi_{k}\colon k \to \frac{k}{\maxOrder} \quad\text{and}\quad
  \phi_{\epsilon(k)}\colon \epsilon(k) \to \frac{\log_{10}(\epsilon(k))}%
    {\lfloor \log_{10}(\epsilon_{\text{min}}) \rfloor}.
\end{equation}
Thus, for a given error measure, the MORscore assigns each data-driven modeling
method a number in the interval $[0, 1]$, where a number near $0$ indicates
large relative errors for all computed orders while a MORscore near $1$
indicates that the method was able to fit the data well for increasing orders.

Besides these error measures, we also report the convergence of the methods
in terms of their optimization objectives.
We display the value of $\lVert \arbErrVec \rVert_{2}^{2}$ at each
iteration $k$ for each method, where $\arbErrVec$ is the residual
vector associated to the method in question.
Specifically, we have that $\arbErrVec = \ek$ for \AAA{},
$\arbErrVec = \ekSO$ for \AAAFull{} and \AAALin{}, and
$\arbErrVec = \ekTrue$ for \AAAFullNL{}.


\subsection{Viscoelastic sandwich beam}%
\label{sec:results_sandwich}

The first numerical example we consider is a beam consisting of a viscoelastic
ethylene-propylene-diene core sandwiched between two layers of cold-rolled
steel~\cite{VanMM13, AumW24}.
The beam is clamped on one side and free to move on the other.
The input is a force on the free end, and the output is the displacement
measured at the same point.
The spatial discretization of the beam results in a dynamical system with
transfer function
\begin{equation} \label{eqn:transferFunction_sandwich}
  H(s) = \mc^{\trans}\left(s^{2} \mM + \mK + \frac{G_{0} + G_{\infty}
    (s \tau)^{\alpha}}{1 + (s \tau)^{\alpha}} \mG \right)^{-1} \mb,
\end{equation}
where $\mM,\mK, \mG \in \R^{3360 \times 3360}$ and $\mb, \mc \in \R^{3360}$. 
The static shear modulus is $G_{0} = 350.4$\,kPa,
the asymptotic shear modulus is $G_{\infty} = 3.062$\,MPa,
the relaxation time is $\tau = 8.230$\,ns, and the fractional parameter is
$\alpha = 0.675$.
We note that~\cref{eqn:transferFunction_sandwich} is \emph{not} a rational
function, and in particular, not in the linear second-order
form~\cref{eqn:MechTransfer}.
However, as is typical of mechanical systems, the model does have explicit
dependence on the second time derivative, as seen by the $s^{2} \mM$ term in
the transfer function~\cref{eqn:transferFunction_sandwich}.
Since our models are constructed purely from data, we are free to produce
linear second-order approximations regardless of the (potentially unknown) true
underlying transfer function, potentially preserving some of the underlying 
structure.

The data for this example is generated by sampling $g_{i} = H(\mu_{i})$ at
$\mu_{i} = \omega_{i}\imunit$, where $\omega_{i}$ are $1\,000$
logarithmically spaced points in the interval $[10^{1}, 10^{4}]$\,rad/s.
The frequency response of the generated data can be seen in
\Cref{fig:resp_sandwich}.
The data weights are chosen to be relative weights so that we have
$\eta_{i} = \lvert g_{i} \rvert^{-1}$.
The original transfer function~\cref{eqn:transferFunction_sandwich} is given by
real matrices and exhibits the same behavior with respect to complex conjugation
as the real second-order models we discussed earlier.
Therefore, we restrict our constructed models to also be real; see
\Cref{sec:realification}.

First, we apply \AAAFull{} to the data with a convergence tolerance on the
relative weighted $\Ltwo$ error~\cref{eqn:ltwoerror} of $0.5\%$.
\AAAFull{} is able to attain this accuracy with an order-$14$
second-order model.
To compare the performance of \AAAFull{} to all other methods, we set
$\maxOrder = 14$ and construct an order-$\maxOrder$ approximation to the data
with \AAA{}, \AAALin{} and \AAAFullNL{}.
The model created by \AAAtwo{} for the comparison has the maximum order
$2 \maxOrder = 28$.
We examine the convergence behavior of each method in
\Cref{fig:errAndOptFun_sandwich}.
Since the models are constructed to be real, they can only attain even orders.

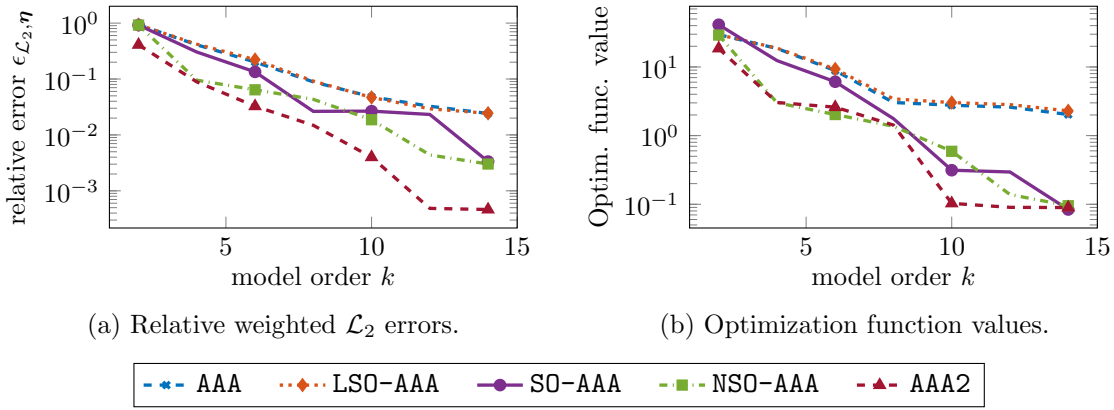
\begin{figure}[t]
  \centering
  \begin{subfigure}[b]{.49\linewidth}
    \centering
  \tikzexternalenable%
  \tikzsetnextfilename{errorPerOrder_sandwich}%
  \begin{tikzpicture}[font = \plotfontsize]
  \pgfplotstableread{graphics/data/sandwich_L2.csv}\tableERR

  \begin{semilogyaxis}[
    scale only axis,
    width              = .73\linewidth,
    height             = .4\linewidth,
    xmin               = 1,
    xmax               = 15,
    xminorticks        = false,
    yminorticks        = true,
    scaled x ticks     = false,
    xlabel             = {model order $k$},
    ylabel             = {relative error $\LtwoWeighted$},
    xlabel style       = {yshift = .3em},
    ylabel style       = {yshift = -.3em},
    x tick label style = {/pgf/number format/1000 sep={\,}},
    y tick label style = {/pgf/number format/1000 sep={\,}}
  ]
    \addplot[AAAConverge] table[x = k, y = AAA]{\tableERR};
    \addplot[LSOAAAConverge] table[x = k, y = SOLAAA]{\tableERR};
    \addplot[SOAAAConverge] table[x = k, y = SOAAA]{\tableERR};
    \addplot[NLSOAAAConverge] table[x = k, y = SONLAAA]{\tableERR};
    \addplot[AAA2kConverge] table[x = k, y = AAA2k]{\tableERR};
  \end{semilogyaxis}
\end{tikzpicture}%
  \tikzexternaldisable%

    \caption{Relative weighted $\Ltwo$ errors.}
    \label{fig:err_sandwich}
  \end{subfigure}%
  \hfill%
  \begin{subfigure}[b]{.49\linewidth}
    \centering
  \tikzexternalenable%
  \tikzsetnextfilename{optFun_sandwich}%
  \begin{tikzpicture}[font = \plotfontsize]
  \pgfplotstableread{graphics/data/sandwich_optFun.csv}\tableFNC

  \begin{semilogyaxis}[
    scale only axis,
    width              = .73\linewidth,
    height             = .4\linewidth,
    xmin               = 1,
    xmax               = 15,
    xminorticks        = false,
    yminorticks        = true,
    scaled x ticks     = false,
    xlabel             = {model order $k$},
    ylabel             = {Optim. func. value},
    xlabel style       = {yshift = .3em},
    ylabel style       = {yshift = -.3em},
    x tick label style = {/pgf/number format/1000 sep={\,}},
    y tick label style = {/pgf/number format/1000 sep={\,}}
  ]

    \addplot[AAAConverge] table[x = k, y = AAA]{\tableFNC};
    \addplot[LSOAAAConverge] table[x = k, y = SOLAAA]{\tableFNC};
    \addplot[SOAAAConverge] table[x = k, y = SOAAA]{\tableFNC};
    \addplot[NLSOAAAConverge] table[x = k, y = SONLAAA]{\tableFNC};
    \addplot[AAA2kConverge] table[x = k, y = AAA2k]{\tableFNC};
  \end{semilogyaxis}
\end{tikzpicture}%
  \tikzexternaldisable%

    \caption{Optimization function values.}
    \label{fig:optFun_sandwich}
  \end{subfigure}
  
  \vspace{.5\baselineskip}
  \tikzexternalenable%
  \tikzsetnextfilename{legend_conv}%
  \begin{tikzpicture}
  \begin{axis}[%
    hide axis,
    scale only axis,
    width  = 1cm,
    height = 1cm,
    xmin   = 0,
    xmax   = 1,
    ymin   = 0,
    ymax   = 1,
    legend columns    = -1,
    legend cell align = {left},
    legend style      = {
      at     = {(0,0)},
      anchor = center,
      /tikz/every even column/.append style = {column sep = 0.4cm}}
  ]
  
    \addlegendimage{AAAConverge} coordinates {(0, 0)};
    \addlegendentry{\AAA}

    \addlegendimage{LSOAAAConverge}
    \addlegendentry{\AAALin}

    \addlegendimage{SOAAAConverge}
    \addlegendentry{\AAAFull}

    \addlegendimage{NLSOAAAConverge}
    \addlegendentry{\AAAFullNL}

    \addlegendimage{AAA2kConverge}
    \addlegendentry{\AAAtwo}
  \end{axis}
\end{tikzpicture}%
  \tikzexternaldisable%

  \caption{Convergence behavior of the different methods for the sandwich beam
    example in terms of the relative weighted $\Ltwo$ approximation errors and
    the optimization function values.
    \AAAFull{} and \AAAFullNL{} show a very similar error behavior, while
    \AAALin{} and \AAA{} are nearly identical.}
  \label{fig:errAndOptFun_sandwich}
\end{figure}

\Cref{fig:err_sandwich} confirms the analysis done in
\Cref{sec:AccuracyHeuristics}.
We observe that both \AAA{} and \AAALin{} have nearly identical errors for each
model order $k$ due to our choice to place the $\sigk$ parameters far in the
left-half plane; we initialize $\real(\sigk_j) = -10^{5}$ for each second-order
method.
Both second-order methods that optimize the quasi-support points $\sigk_{j}$
have smaller approximation errors for each order $k$ than \AAALin{},
with \AAAFullNL{} generally performing better due to minimizing the nonlinear
error as opposed to the separable error that is minimized by \AAAFull{}.
Also, we see that the unstructured \AAAtwo{} approximation, which is a
rational function of the same degree as constructed by the second-order
methods, performs the best at each iteration $k$.
This is the expected behavior since the second-order model structure is not
enforced in the \AAAtwo{} approximations.

In \Cref{fig:optFun_sandwich}, we show the value of the individual optimization
functions for each method throughout the iteration as described in
\Cref{sec:setup}.
Interestingly, we observe that while the separable error $\ekSO$ decreases
significantly from $k = 8$ to $k = 10$, the actual error value $\LtwoWeighted$
stays approximately constant.
This is an artifact resulting from the approximation made to obtain the
separable form of $\ekSO$.
Since $\ekSO$ ignores the term $1 / \denomSOk$ in the error expression,
minimizing $\ekSO$ may not decrease $\LtwoWeighted$.
Predicting and quantifying the effects of the deviation of the separable
(or fully linear) error vectors from the true nonlinear error is an
interesting direction for future research.
For now, we emphasize that while convergence in the relative weighted $\Ltwo$
error, $\LtwoWeighted$, may stagnate for a few iterations or even be
non-monotonic, we observe that \AAA{}, \AAALin{} and \AAAFull{} all have
satisfactorily low $\LtwoWeighted$ errors after the prescribed number of
iterations.

For further comparisons of the performance of the second-order AAA methods, the
MORscores of the second-order methods in each of the error measures
$\LtwoWeighted$, $\LinfWeighted$ and $\relerr$ are provided in
\Cref{tab:MORscores_sandwich}.
The MORscores show that the trend showcased in \Cref{fig:err_sandwich} for the
$\LtwoWeighted$ error continues across each error measure.
In particular, both second-order methods that optimize the quasi-support points
(\AAAFull{} and \AAAFullNL{}) perform significantly better in each error metric
than the fully linearized \AAALin{}.
Also, the fully nonlinear \AAAFullNL{} outperforms \AAAFull{}.
The difference in MORscores between \AAAFullNL{} and \AAAFull{} is due to the
misalignment between the optimization function $\ekSO$ and the true error
metrics.
This can be seen in \Cref{fig:errAndOptFun_sandwich} as the \AAAFull{}
approximations do not decrease the $\LtwoWeighted$ error measure for the
orders $k = 6$ to $k = 10$.

\begin{table}[t]
    \centering
    \caption{MORscores of the second-order methods for the sandwich beam example
      with minimum attainable tolerance $\epsilon_{\min} = 10^{-8}$.
      Due to the chosen weights $\eta_{i}$, the $\Linf$ and pointwise error
      measures are identical $\LinfWeighted \equiv \relerr$.
      We see that \AAAFullNL{} is the best performing method followed by
      \AAAFull{} and then \AAALin{}.}
    \label{tab:MORscores_sandwich}

    \vspace{.5\baselineskip}
    \begin{tabular}{lrrr}
    \hline\noalign{\smallskip}
    \multicolumn{1}{c}{\textbf{Algorithm}} &
      \multicolumn{1}{c}{$\LtwoWeighted$} &
      \multicolumn{1}{c}{$\LinfWeighted$} &
      \multicolumn{1}{c}{$\relerr$} \\
    \noalign{\smallskip}\hline\noalign{\smallskip}
    \AAALin & 0.103 & 0.062 & 0.062 \\
    \AAAFull & 0.133 & 0.081 & 0.081 \\
    \AAAFullNL & 0.160 & 0.120 & 0.120 \\
    \noalign{\smallskip}\hline\noalign{\smallskip}
    \end{tabular}
\end{table}

Finally, we examine the frequency response of each converged model and their
associated pointwise weighted errors, $\ekResults$, in
\Cref{fig:respAndErrVsFreq_sandwich}.
The frequency response of each model is indistinguishable from the data, as
desired.
From the pointwise weighted errors displayed in \Cref{fig:ptwiseErr_sandwich},
we observe that the error of the \AAAFullNL{} model is fairly constant across
the whole frequency range.
This consistent error is likely a result of \AAAFullNL{} minimizing the true
nonlinear approximation error.
The errors of the other methods vary significantly more with the frequency
likely due to the approximate least-squares formulations that are used in
addition to the interpolation.

\begin{figure}[t]
  \centering
  \begin{subfigure}[b]{.49\linewidth}
    \centering
  \tikzexternalenable%
  \tikzsetnextfilename{resp_sandwich}%
  \begin{tikzpicture}[font = \plotfontsize]
  \pgfplotstableread{graphics/data/sandwich_response.csv}\tableRESP

  \begin{loglogaxis}[
    scale only axis,
    width              = .7\linewidth,
    height             = .4\linewidth,
    xmin               = 1e+1,
    xmax               = 1e+4,
    ymin               = 1e-3,
    ymax               = 1e-1,
    xminorticks        = true,
    yminorticks        = true,
    scaled x ticks     = false,
    clip mode          = individual,
    xlabel             = {frequency $\omega$ (rad/s)},
    ylabel             = {magnitude $\lvert H(\omega \imunit) \rvert$},
    xlabel style       = {yshift = .3em},
    ylabel style       = {yshift = -.3em},
    x tick label style = {/pgf/number format/1000 sep={\,}},
    y tick label style = {/pgf/number format/1000 sep={\,}}
  ]
    \addplot[trueData] table[x = mu, y = g]{\tableRESP};
    \addplot[AAAResponse] table[x = mu, y = AAA]{\tableRESP};
    \addplot[LSOAAAResponse] table[x = mu, y = SOLAAA]{\tableRESP};
    \addplot[SOAAAResponse] table[x = mu, y = SOAAA]{\tableRESP};
    \addplot[NLSOAAAResponse] table[x = mu, y = SONLAAA]{\tableRESP};
    \addplot[AAA2kResponse] table[x = mu, y = AAA2k]{\tableRESP};
  \end{loglogaxis}
\end{tikzpicture}%
  \tikzexternaldisable%

    \caption{Frequency response.}
    \label{fig:resp_sandwich}
  \end{subfigure}%
  \hfill%
  \begin{subfigure}[b]{.49\linewidth}
    \centering
  \tikzexternalenable%
  \tikzsetnextfilename{errPerFreq_sandwich}%
  \begin{tikzpicture}[font = \plotfontsize]
  \pgfplotstableread{graphics/data/sandwich_error.csv}\tableERR

  \begin{loglogaxis}[
    scale only axis,
    width              = .7\linewidth,
    height             = .4\linewidth,
    xmin               = 1e+1,
    xmax               = 1e+4,
    ymin               = 3e-7,
    ymax               = 2e-1,
    xminorticks        = true,
    yminorticks        = true,
    scaled x ticks     = false,
    clip mode          = individual,
    xlabel             = {frequency $\omega$ (rad/s)},
    ylabel             = {pointwise error $\ekResults$},
    xlabel style       = {yshift = .3em},
    ylabel style       = {yshift = -.3em},
    x tick label style = {/pgf/number format/1000 sep={\,}},
    y tick label style = {/pgf/number format/1000 sep={\,}}
  ]
    \addplot[AAAResponse] table[x = mu, y = AAA]{\tableERR};
    \addplot[LSOAAAResponse] table[x = mu, y = SOLAAA]{\tableERR};
    \addplot[SOAAAResponse] table[x = mu, y = SOAAA]{\tableERR};
    \addplot[NLSOAAAResponse] table[x = mu, y = SONLAAA]{\tableERR};
    \addplot[AAA2kResponse] table[x = mu, y = AAA2k]{\tableERR};
  \end{loglogaxis}
\end{tikzpicture}%
  \tikzexternaldisable%

    \caption{Pointwise weighted error.}
    \label{fig:ptwiseErr_sandwich}
  \end{subfigure}

  \vspace{.5\baselineskip}
  \tikzexternalenable%
  \tikzsetnextfilename{legend_freq}%
  \begin{tikzpicture}
  \begin{axis}[%
    hide axis,
    scale only axis,
    width  = 1cm,
    height = 1cm,
    xmin   = 0,
    xmax   = 1,
    ymin   = 0,
    ymax   = 1,
    legend columns    = -1,
    legend cell align = {left},
    legend style      = {
      at     = {(0,0)},
      anchor = center,
      /tikz/every even column/.append style = {column sep = 0.4cm}}
  ]

    \addplot[trueData] coordinates {(0, 0)};
    \addlegendentry{Original data}

    \addlegendimage{AAAResponse}
    \addlegendentry{\AAA}

    \addlegendimage{LSOAAAResponse}
    \addlegendentry{\AAALin}

    \addlegendimage{SOAAAResponse}
    \addlegendentry{\AAAFull}

    \addlegendimage{NLSOAAAResponse}
    \addlegendentry{\AAAFullNL}

    \addlegendimage{AAA2kResponse}
    \addlegendentry{\AAAtwo}
  \end{axis}
\end{tikzpicture}%
  \tikzexternaldisable%

  \caption{Frequency response and pointwise weighted errors for the
    sandwich beam example with $\maxOrder = 14$ models.
    All methods provide reasonably accurate approximations and can reproduce
    the given data.}
  \label{fig:respAndErrVsFreq_sandwich}
\end{figure}
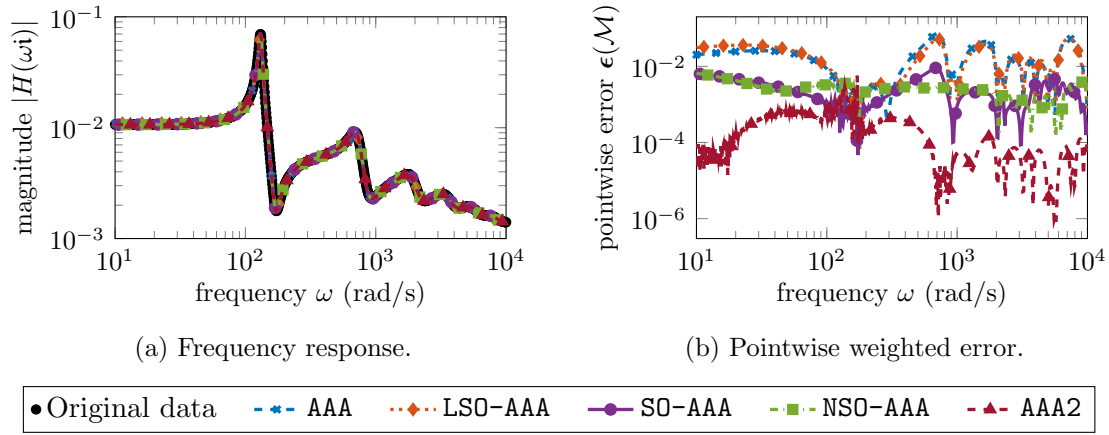


\subsection{Butterfly gyroscope}%
\label{sec:results_gyro}

The next example describes the vibrational displacement of a butterfly
gyroscope~\cite{Bil05, morwiki_gyro}.
The butterfly gyroscope is an inertial navigation device with four inertial
sensors, each of which detects displacements in the three spatial directions,
leading to a total of $12$ outputs.
To adapt the model to our single-input/single-output setting,
we take the average of the displacements into the y-direction.
The resulting model has the second-order form~\cref{eqn:MechSystem}
with transfer function~\cref{eqn:MechTransfer}.
The model is real so that for this example we constrain our data-driven models
to have real state-space representations; see \Cref{sec:realification}.

As data for this example, we begin with frequency response samples
$g_{i} = H(\mu_{i})$ at $\mu_{i} = \imunit\omega_{i}$, where $\omega_{i}$
are $500$ logarithmically distributed points in $[10^{3}, 10^{6}]$\,rad/s.
Near the end of this interval, the magnitude of the model's frequency response
decays rapidly.
For numerical robustness, we restrict the data to frequency response samples
with magnitude larger than $10^{-6}$, resulting in overall $408$ data points.
The weights $\eta_{i}$ are taken to be the reciprocal of the transfer function
magnitude, that is, we have $\eta_{i} = \lvert g_{i} \rvert^{-1}$ modeling
the relative weighting of the given data.

We specify a convergence tolerance on the $\LtwoWeighted$ error of the
\AAAFull{} model of $10^{-6}$, which results in an order-$22$ second-order
model.
To compare \AAAFull's convergence to the other AAA methods, we set
$\maxOrder = 22$ and compute order-$\maxOrder$ approximations via
\AAA{}, \AAALin{}, \AAAFullNL{} and \AAAtwo{}.
The convergence of each method is illustrated in \Cref{fig:errAndOptFun_gyro}.
Note that as for the previous example, only even orders can be shown shown
since the constructed models are enforced to be real.

\begin{figure}[t]
  \centering
  \begin{subfigure}[b]{.49\linewidth}
    \centering
  \tikzexternalenable%
  \tikzsetnextfilename{errorPerOrder_gyro}%
  \begin{tikzpicture}[font = \plotfontsize]
  \pgfplotstableread{graphics/data/gyroY_L2.csv}\tableERR

  \begin{semilogyaxis}[
    scale only axis,
    width              = .73\linewidth,
    height             = .4\linewidth,
    xmin               = 1,
    xmax               = 23,
    xminorticks        = false,
    yminorticks        = true,
    scaled x ticks     = false,
    xlabel             = {model order $k$},
    ylabel             = {relative error $\LtwoWeighted$},
    xlabel style       = {yshift = .3em},
    ylabel style       = {yshift = -.3em},
    x tick label style = {/pgf/number format/1000 sep={\,}},
    y tick label style = {/pgf/number format/1000 sep={\,}}
  ]
    \addplot[AAAConverge] table[x = k, y = AAA]{\tableERR};
    \addplot[LSOAAAConverge] table[x = k, y = SOLAAA]{\tableERR};
    \addplot[SOAAAConverge] table[x = k, y = SOAAA]{\tableERR};
    \addplot[NLSOAAAConverge] table[x = k, y = SONLAAA]{\tableERR};
    \addplot[AAA2kConverge] table[x = k, y = AAA2k]{\tableERR};
  \end{semilogyaxis}
\end{tikzpicture}%
  \tikzexternaldisable%

    \caption{Relative weighted $\Ltwo$ errors.}
    \label{fig:err_gyro}
  \end{subfigure}%
  \hfill%
  \begin{subfigure}[b]{.49\linewidth}
    \centering
  \tikzexternalenable%
  \tikzsetnextfilename{optFun_gyro}%
  \begin{tikzpicture}
  \pgfplotstableread{graphics/data/gyroY_optFun.csv}\tableFNC

  \begin{semilogyaxis}[
    scale only axis,
    width              = .73\linewidth,
    height             = .4\linewidth,
    xmin               = 1,
    xmax               = 23,
    xminorticks        = false,
    yminorticks        = true,
    scaled x ticks     = false,
    xlabel             = {model order $k$},
    ylabel             = {Optim. func. value},
    xlabel style       = {yshift = .3em},
    ylabel style       = {yshift = -.3em},
    x tick label style = {/pgf/number format/1000 sep={\,}},
    y tick label style = {/pgf/number format/1000 sep={\,}}
  ]

    \addplot[AAAConverge] table[x = k, y = AAA]{\tableFNC};
    \addplot[LSOAAAConverge] table[x = k, y = SOLAAA]{\tableFNC};
    \addplot[SOAAAConverge] table[x = k, y = SOAAA]{\tableFNC};
    \addplot[NLSOAAAConverge] table[x = k, y = SONLAAA]{\tableFNC};
    \addplot[AAA2kConverge] table[x = k, y = AAA2k]{\tableFNC};
  \end{semilogyaxis}
\end{tikzpicture}%
  \tikzexternaldisable%

    \caption{Optimization function values.}
    \label{fig:optFun_gyro}
  \end{subfigure}

  \vspace{.5\baselineskip}
  \tikzexternalenable%
  \tikzsetnextfilename{legend_conv}%
  \begin{tikzpicture}
  \begin{axis}[%
    hide axis,
    scale only axis,
    width  = 1cm,
    height = 1cm,
    xmin   = 0,
    xmax   = 1,
    ymin   = 0,
    ymax   = 1,
    legend columns    = -1,
    legend cell align = {left},
    legend style      = {
      at     = {(0,0)},
      anchor = center,
      /tikz/every even column/.append style = {column sep = 0.4cm}}
  ]
  
    \addlegendimage{AAAConverge} coordinates {(0, 0)};
    \addlegendentry{\AAA}

    \addlegendimage{LSOAAAConverge}
    \addlegendentry{\AAALin}

    \addlegendimage{SOAAAConverge}
    \addlegendentry{\AAAFull}

    \addlegendimage{NLSOAAAConverge}
    \addlegendentry{\AAAFullNL}

    \addlegendimage{AAA2kConverge}
    \addlegendentry{\AAAtwo}
  \end{axis}
\end{tikzpicture}%
  \tikzexternaldisable%

  \caption{Convergence behavior of the different methods for the gyroscope
    example in terms of the relative weighted $\Ltwo$ approximation errors and
    the optimization function values.
    \AAAFull{} and \AAAFullNL{} show a very similar error behavior, while
    \AAALin{} and \AAA{} are nearly identical.}
  \label{fig:errAndOptFun_gyro}
\end{figure}
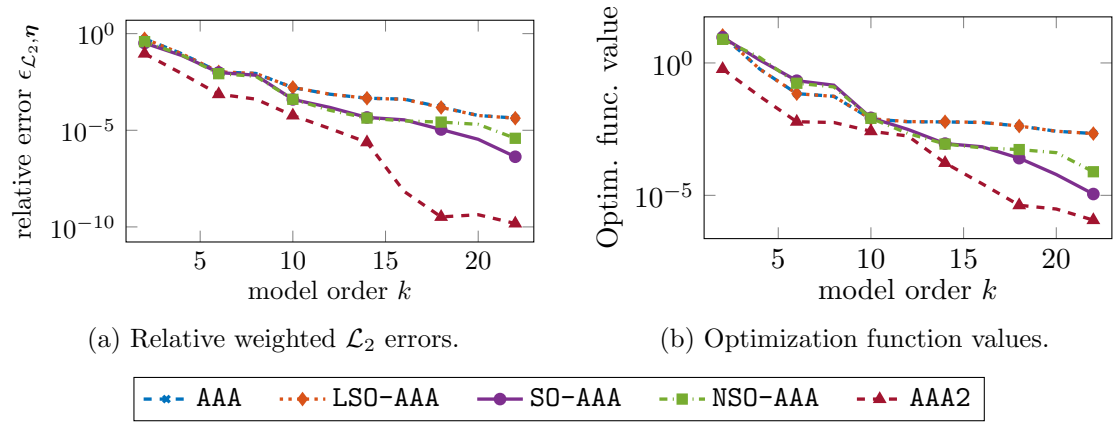

In \Cref{fig:err_gyro}, we see that both linear methods, \AAA{} and \AAALin{},
are indistinguishable in terms of performance, while the second-order AAA
methods that optimize the quasi-support points produce visibly better
approximations.
As for the previous example, \AAAtwo{} is the best performing algorithm.
In contrast to the results in \Cref{sec:results_sandwich}, we see that
\AAAFull{} performs equally well or even better compared to \AAAFullNL{} 
for each model order $k$.
This improved performance is also reflected for each error measure in the
MORscores shown in \Cref{tab:MORscores_gyroscope}.
In \Cref{fig:optFun_gyro}, we see that a reduction the optimization function
value correlates with a reduction in the $\LtwoWeighted$ error in this example.
This observation is also in contrast to the results in
\Cref{sec:results_sandwich}, where a reduction in the optimization function
value was not correlated with a reduction in the $\LtwoWeighted$ error.
These two observations indicate that in this example, the separable and
fully linear residual vectors $\ekSO$ are a high-fidelity proxy for the true,
fully nonlinear residual $\ekTrueSO$.
A different possibility is that the nonlinear optimization process in
\AAAFullNL{} struggles to find suitable approximation so that the optimization
in \AAAFull{} involving a smaller set of parameters performs better.
Due to the severe nonlinearities of the fully nonlinear error, we expect more
local minima, which may have errors much larger than the global minimum. 
Additionally, due to the wide frequency range considered for this problem and
the large deviations in transfer function magnitudes, the numerical computation
of the associated gradients may be inaccurate, resulting in insufficient
directions for the optimization.

\begin{table}[t]
    \centering
    \caption{MORscores of the second-order methods for the gyroscope example
      with minimum attainable tolerance $\epsilon_{\min} = 10^{-8}$.
      Due to the chosen weights $\eta_{i}$ the $\Linf$ and pointwise error
      measures are identical $\LinfWeighted \equiv \relerr$.
      In this example, \AAAFull{} is the best performing method followed by
      \AAAFullNL{} and then \AAALin{}.}
    \label{tab:MORscores_gyroscope}

    \vspace{.5\baselineskip}
    \begin{tabular}{lrrr}
    \hline\noalign{\smallskip}
    \multicolumn{1}{c}{\textbf{Algorithm}} &
      \multicolumn{1}{c}{$\LtwoWeighted$} &
      \multicolumn{1}{c}{$\LinfWeighted$} &
      \multicolumn{1}{c}{$\relerr$} \\
    \noalign{\smallskip}\hline\noalign{\smallskip}
    \AAALin & 0.319 & 0.235 & 0.235 \\
    \AAAFull & 0.400 & 0.304 & 0.304 \\
    \AAAFullNL & 0.384 & 0.285 & 0.285 \\
    \noalign{\smallskip}\hline\noalign{\smallskip}
    \end{tabular}
\end{table}

\begin{figure}[t]
  \centering
  \begin{subfigure}[b]{.49\linewidth}
    \centering
  \tikzexternalenable%
  \tikzsetnextfilename{resp_gyro}%
  \begin{tikzpicture}[font = \plotfontsize]
  \pgfplotstableread{graphics/data/gyroY_response.csv}\tableRESP

  \begin{loglogaxis}[
    scale only axis,
    width              = .7\linewidth,
    height             = .4\linewidth,
    xmin               = 1e+3,
    xmax               = 2.8e+5,
    ymin               = 9e-7,
    ymax               = 5e-3,
    xminorticks        = true,
    yminorticks        = true,
    scaled x ticks     = false,
    clip mode          = individual,
    xlabel             = {frequency $\omega$ (rad/s)},
    ylabel             = {magnitude $\lvert H(\omega \imunit) \rvert$},
    xlabel style       = {yshift = .3em},
    ylabel style       = {yshift = -.3em},
    x tick label style = {/pgf/number format/1000 sep={\,}},
    y tick label style = {/pgf/number format/1000 sep={\,}}
  ]
    \addplot[trueData] table[x = mu, y = g]{\tableRESP};
    \addplot[AAAResponse] table[x = mu, y = AAA]{\tableRESP};
    \addplot[LSOAAAResponse] table[x = mu, y = SOLAAA]{\tableRESP};
    \addplot[SOAAAResponse] table[x = mu, y = SOAAA]{\tableRESP};
    \addplot[NLSOAAAResponse] table[x = mu, y = SONLAAA]{\tableRESP};
    \addplot[AAA2kResponse] table[x = mu, y = AAA2k]{\tableRESP};
  \end{loglogaxis}
\end{tikzpicture}%
  \tikzexternaldisable%

    \caption{Frequency response.}
    \label{fig:resp_gyro}
  \end{subfigure}%
  \hfill%
  \begin{subfigure}[b]{.49\linewidth}
    \centering
  \tikzexternalenable%
  \tikzsetnextfilename{errPerFreq_gyro}%
  \begin{tikzpicture}[font = \plotfontsize]
  \pgfplotstableread{graphics/data/gyroY_error.csv}\tableERR

  \begin{loglogaxis}[
    scale only axis,
    width              = .7\linewidth,
    height             = .4\linewidth,
    xmin               = 1e+3,
    xmax               = 2.8e+5,
    ymin               = 1e-12,
    ymax               = 1e-3,
    xminorticks        = true,
    yminorticks        = true,
    scaled x ticks     = false,
    clip mode          = individual,
    xlabel             = {frequency $\omega$ (rad/s)},
    ylabel             = {pointwise error $\ekResults$},
    xlabel style       = {yshift = .3em},
    ylabel style       = {yshift = -.3em},
    x tick label style = {/pgf/number format/1000 sep={\,}},
    y tick label style = {/pgf/number format/1000 sep={\,}}
  ]
    \addplot[AAAResponse] table[x = mu, y = AAA]{\tableERR};
    \addplot[LSOAAAResponse] table[x = mu, y = SOLAAA]{\tableERR};
    \addplot[SOAAAResponse] table[x = mu, y = SOAAA]{\tableERR};
    \addplot[NLSOAAAResponse] table[x = mu, y = SONLAAA]{\tableERR};
    \addplot[AAA2kResponse] table[x = mu, y = AAA2k]{\tableERR};
  \end{loglogaxis}
\end{tikzpicture}%
  \tikzexternaldisable%

    \caption{Pointwise weighted error.}
    \label{fig:ptwiseErr_gyro}
  \end{subfigure}
  
  \vspace{.5\baselineskip}
  \tikzexternalenable%
  \tikzsetnextfilename{legend_freq}%
  \begin{tikzpicture}
  \begin{axis}[%
    hide axis,
    scale only axis,
    width  = 1cm,
    height = 1cm,
    xmin   = 0,
    xmax   = 1,
    ymin   = 0,
    ymax   = 1,
    legend columns    = -1,
    legend cell align = {left},
    legend style      = {
      at     = {(0,0)},
      anchor = center,
      /tikz/every even column/.append style = {column sep = 0.4cm}}
  ]

    \addplot[trueData] coordinates {(0, 0)};
    \addlegendentry{Original data}

    \addlegendimage{AAAResponse}
    \addlegendentry{\AAA}

    \addlegendimage{LSOAAAResponse}
    \addlegendentry{\AAALin}

    \addlegendimage{SOAAAResponse}
    \addlegendentry{\AAAFull}

    \addlegendimage{NLSOAAAResponse}
    \addlegendentry{\AAAFullNL}

    \addlegendimage{AAA2kResponse}
    \addlegendentry{\AAAtwo}
  \end{axis}
\end{tikzpicture}%
  \tikzexternaldisable%

  \caption{Frequency response and pointwise weighted errors for the gyroscope
    example with $\maxOrder = 22$ models.
    The proposed second-order AAA methods can accurately reproduce the
    given data while enforcing the desired system structure.}
  \label{fig:RespAndErrVsFreq_gyro}
\end{figure}
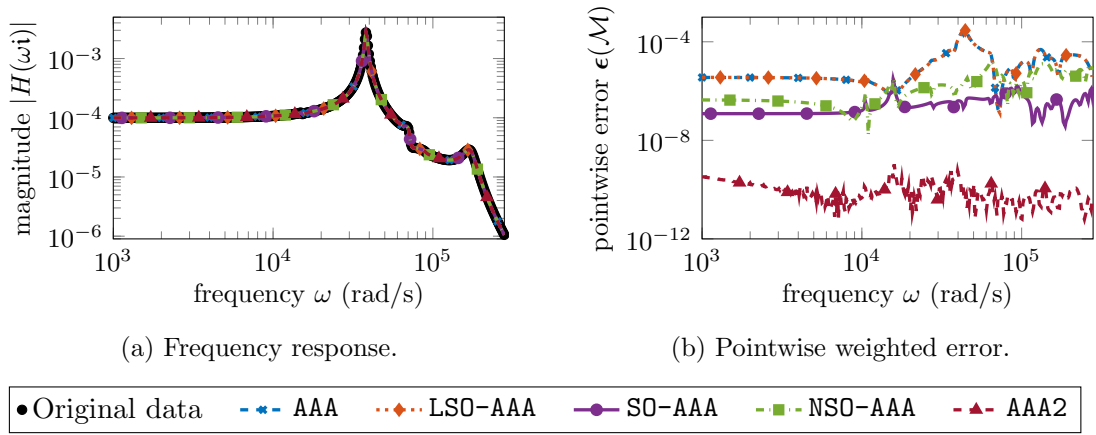

Finally, the quality of the models of order $\maxOrder = 22$ constructed by the
different methods is shown in form of the frequency response and error in
\Cref{fig:RespAndErrVsFreq_gyro}.
In the error plot  \Cref{fig:ptwiseErr_gyro}, we observe that in this example,
the two methods \AAAFull{} and \AAAtwo{} result in reasonably consistent errors,
while \AAAFullNL{} has a slight increase in error with frequency.
As in the previous example, the models constructed by the linear methods \AAA{}
and \AAALin{} have errors with a higher frequency dependence.
\AAAFullNL{}'s struggle in the high frequency range supports the hypothesis
that the method may be plagued by numerical errors due to the large
frequency values $\lvert \mu_{i} \rvert$ and the small transfer function
values $\lvert g_{i} \rvert$.
\AAAFull{}'s comparatively excellent performance in this range is an
encouraging indicator for the method's practical performance.


\subsection{Acoustic cavity with poroelastic layer}%
\label{sec:poracExample}

As the final numerical example, we consider the model of an acoustic cavity
with dimensions $0.75\times 0.6 \times 0.4$\,m, where the wall opposite of the
sound source is covered by a $0.05$\,m layer of poroelastic coating acting as a
sound absorber.
The model's damping is frequency dependent, leading to a frequency dependent
transfer function of the general form
\begin{equation} \label{eqn:cavityTF}
  H(s) = \mc^{\trans} \left( s^{2} \mM + \mK +
    \sum\limits_{i = 1}^{\ell} \phi_{i}(s) \mG_{i} \right)^{-1} \mb.
\end{equation}
The model was originally introduced in~\cite{RumGD14}, and it was implemented
in~\cite{AumW23} using a finite element discretization of order
$n = 386\,076$.
Similar to the example in \Cref{sec:results_sandwich}, we can see
how~\cref{eqn:cavityTF} resembles parts of the second-order transfer function
structure~\cref{eqn:MechTransfer} so that we expect to achieve good
approximations using our proposed methods.

For the data, we sample the transfer function~\cref{eqn:cavityTF} using
\begin{equation}
    \mu_{i} = \imunit\omega_{i}
    \quad\text{for}\quad
    \omega_{i} = 2\pi(i + 100),
    \quad\text{with}\quad
    i = 0, 1, \ldots, 900,
\end{equation}
resulting in $901$ frequency samples in the interval $[100, 1000]$\,Hz.
Due to the acoustic nature of this model, the underlying matrices are complex. 
Thus, in contrast to both previous examples, we do not restrict our our data
driven models to be real.
Additionally, we do not weigh the given data for this example so that we have
$\eta_{i} = 1$.
We seek a \AAAFull{} to fit the data with a $\LtwoWeighted$ error of less
than $0.5\%$, which results in an order-$25$ model.
For the comparison, we set $\maxOrder = 25$ to compute approximations with all
other AAA methods.
We examine the convergence of the methods in \Cref{fig:errAndOptFun_porac}.

\begin{figure}[t]
  \centering
  \begin{subfigure}[b]{.49\linewidth}
    \centering
  \tikzexternalenable%
  \tikzsetnextfilename{errorPerOrder_poro}%
  \begin{tikzpicture}[font = \plotfontsize]
  \pgfplotstableread{graphics/data/porac_L2.csv}\tableERR

  \begin{semilogyaxis}[
    scale only axis,
    width              = .73\linewidth,
    height             = .4\linewidth,
    xmin               = 1,
    xmax               = 26,
    xminorticks        = false,
    yminorticks        = true,
    scaled x ticks     = false,
    xlabel             = {model order $k$},
    ylabel             = {relative error $\LtwoWeighted$},
    xlabel style       = {yshift = .3em},
    ylabel style       = {yshift = -.3em},
    x tick label style = {/pgf/number format/1000 sep={\,}},
    y tick label style = {/pgf/number format/1000 sep={\,}}
  ]
    \addplot[AAAConverge] table[x = k, y = AAA]{\tableERR};
    \addplot[LSOAAAConverge] table[x = k, y = SOLAAA]{\tableERR};
    \addplot[SOAAAConverge] table[x = k, y = SOAAA]{\tableERR};
    \addplot[NLSOAAAConverge] table[x = k, y = SONLAAA]{\tableERR};
    \addplot[AAA2kConverge] table[x = k, y = AAA2k]{\tableERR};
  \end{semilogyaxis}
\end{tikzpicture}%
  \tikzexternaldisable%

    \caption{$\LtwoWeighted$ error}
    \label{fig:err_porac}
  \end{subfigure}%
  \hfill%
  \begin{subfigure}[b]{.49\linewidth}
    \centering
  \tikzexternalenable%
  \tikzsetnextfilename{optFun_porac}%
  \begin{tikzpicture}
  \pgfplotstableread{graphics/data/porac_optFun.csv}\tableFNC

  \begin{semilogyaxis}[
    scale only axis,
    width              = .73\linewidth,
    height             = .4\linewidth,
    xmin               = 1,
    xmax               = 26,
    xminorticks        = false,
    yminorticks        = true,
    scaled x ticks     = false,
    xlabel             = {model order $k$},
    ylabel             = {Optim. func. value},
    xlabel style       = {yshift = .3em},
    ylabel style       = {yshift = -.3em},
    x tick label style = {/pgf/number format/1000 sep={\,}},
    y tick label style = {/pgf/number format/1000 sep={\,}}
  ]

    \addplot[AAAConverge] table[x = k, y = AAA]{\tableFNC};
    \addplot[LSOAAAConverge] table[x = k, y = SOLAAA]{\tableFNC};
    \addplot[SOAAAConverge] table[x = k, y = SOAAA]{\tableFNC};
    \addplot[NLSOAAAConverge] table[x = k, y = SONLAAA]{\tableFNC};
    \addplot[AAA2kConverge] table[x = k, y = AAA2k]{\tableFNC};
  \end{semilogyaxis}
\end{tikzpicture}%
  \tikzexternaldisable%

    \caption{Optimization function}
    \label{fig:optFun_porac}
  \end{subfigure}

  \vspace{.5\baselineskip}
  \tikzexternalenable%
  \tikzsetnextfilename{legend_conv}%
  \begin{tikzpicture}
  \begin{axis}[%
    hide axis,
    scale only axis,
    width  = 1cm,
    height = 1cm,
    xmin   = 0,
    xmax   = 1,
    ymin   = 0,
    ymax   = 1,
    legend columns    = -1,
    legend cell align = {left},
    legend style      = {
      at     = {(0,0)},
      anchor = center,
      /tikz/every even column/.append style = {column sep = 0.4cm}}
  ]
  
    \addlegendimage{AAAConverge} coordinates {(0, 0)};
    \addlegendentry{\AAA}

    \addlegendimage{LSOAAAConverge}
    \addlegendentry{\AAALin}

    \addlegendimage{SOAAAConverge}
    \addlegendentry{\AAAFull}

    \addlegendimage{NLSOAAAConverge}
    \addlegendentry{\AAAFullNL}

    \addlegendimage{AAA2kConverge}
    \addlegendentry{\AAAtwo}
  \end{axis}
\end{tikzpicture}%
  \tikzexternaldisable%

  \caption{Convergence behavior of the different methods for the acoustic
    cavity example in terms of the relative weighted $\Ltwo$ approximation
    errors and the optimization function values.
    All three methods, \AAAFull{}, \AAAFullNL{} and \AAAtwo{}, perform very
    similar, while \AAALin{} and \AAA{} are nearly identical.}
  \label{fig:errAndOptFun_porac}
\end{figure}
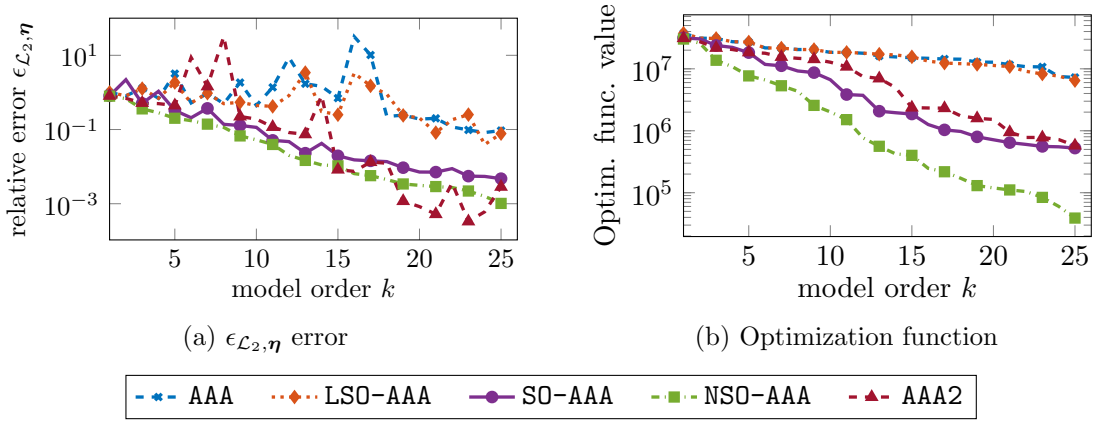

\Cref{fig:err_porac} illustrates that in this example, our second-order methods
are capable of outperforming \AAA{} even with twice the order of approximation
(\AAAtwo{}).
In particular, \AAAFullNL{} performs significantly better than \AAAtwo{} up to
$k = 15$, with \AAAFull{} in a similar region.
We also observe a more severe instance of the lack of corelation between the
optimization function values displayed in \Cref{fig:optFun_porac} and the
true approximation error in \Cref{fig:err_porac}.
In this example, we see large error increases for each of the fully linear
methods (\AAA, \AAALin and \AAAtwo) at several values of $k$, even when the
associated optimization function value decreases.
The convergence of \AAAFull{} is also non-monotonic in this example, though
much less severely than for the fully linear methods.
\AAAFullNL{} is the only method able to monotonically decrease the
$\LtwoWeighted$ error at each value of $k$ due to directly optimizing the true
nonlinear error.
These observations are also reflected in the MORscores presented in
\Cref{tab:MORscores_porac}.

\begin{table}[t]
    \centering
    \caption{MORscores of the second-order methods for the acoustic cavity
      example with minimum attainable tolerance $\epsilon_{\min} = 10^{-8}$.
      In this example, \AAAFullNL{} is the best performing method followed by
      \AAAFull{} and then \AAALin{}.}
    \label{tab:MORscores_porac}

    \vspace{.5\baselineskip}
    \begin{tabular}{lrrr}
    \hline\noalign{\smallskip}
    \multicolumn{1}{c}{\textbf{Algorithm}} &
      \multicolumn{1}{c}{$\LtwoWeighted$} &
      \multicolumn{1}{c}{$\LinfWeighted$} &
      \multicolumn{1}{c}{$\relerr$} \\
    \noalign{\smallskip}\hline\noalign{\smallskip}
    \AAALin & 0.044 & 0.043 & 0.000 \\
    \AAAFull & 0.158 & 0.170 & 0.038 \\
    \AAAFullNL & 0.198 & 0.215 & 0.050 \\
    \noalign{\smallskip}\hline\noalign{\smallskip}
    \end{tabular}
\end{table}

The frequency responses and pointwise weighted errors of the models of size
$\maxOrder = 25$ are presented in \Cref{fig:RespAndErrVsFreq_porac}.
In this example, we notice visible deviations of the frequency responses of the
linear methods from the given data, while the second-order methods that
utilize nonlinear optimization are indistinguishable from the data.
As in \Cref{sec:results_sandwich}, we observe that the error of the
\AAAFullNL{} model appears to have only a small dependence on frequency, while
each method that optimizes an approximation of the true error has large
variations in their errors.
These results indicate that \AAAFullNL{}'s failure to provide highly accurate
approximations in the gyroscope example (\Cref{sec:results_gyro}) is likely
due to the large variations in the magnitude of the transfer function values
because both the poroacoustic and sandwich beam examples have significantly less
variations in the magnitude of their frequency responses.
    
\begin{figure}[t]
  \centering
  \begin{subfigure}[b]{.49\linewidth}
    \centering
  \tikzexternalenable%
  \tikzsetnextfilename{resp_porac}%
  \begin{tikzpicture}[font = \plotfontsize]
  \pgfplotstableread{graphics/data/porac_response.csv}\tableRESP

  \begin{loglogaxis}[
    scale only axis,
    width              = .7\linewidth,
    height             = .4\linewidth,
    xmin               = 628.31,
    xmax               = 6283.19,
    ymin               = 5e+4,
    ymax               = 1e+7,
    xminorticks        = true,
    yminorticks        = true,
    scaled x ticks     = false,
    clip mode          = individual,
    xlabel             = {frequency $\omega$ (rad/s)},
    ylabel             = {magnitude $\lvert H(\omega \imunit) \rvert$},
    xlabel style       = {yshift = .3em},
    ylabel style       = {yshift = -.3em},
    x tick label style = {/pgf/number format/1000 sep={\,}},
    y tick label style = {/pgf/number format/1000 sep={\,}}
  ]
    \addplot[trueData] table[x = mu, y = g]{\tableRESP};
    \addplot[AAAResponse] table[x = mu, y = AAA]{\tableRESP};
    \addplot[LSOAAAResponse] table[x = mu, y = SOLAAA]{\tableRESP};
    \addplot[SOAAAResponse] table[x = mu, y = SOAAA]{\tableRESP};
    \addplot[NLSOAAAResponse] table[x = mu, y = SONLAAA]{\tableRESP};
    \addplot[AAA2kResponse] table[x = mu, y = AAA2k]{\tableRESP};
  \end{loglogaxis}
\end{tikzpicture}%
  \tikzexternaldisable%

    \caption{Frequency response}
    \label{fig:resp_porac}
  \end{subfigure}%
  \hfill%
  \begin{subfigure}[b]{.49\linewidth}
    \centering
  \tikzexternalenable%
  \tikzsetnextfilename{errPerFreq_porac}%
  \begin{tikzpicture}[font = \plotfontsize]
  \pgfplotstableread{graphics/data/porac_error.csv}\tableERR

  \begin{loglogaxis}[
    scale only axis,
    width              = .7\linewidth,
    height             = .4\linewidth,
    xmin               = 628.31,
    xmax               = 6283.19,
    ymin               = 2e-3,
    ymax               = 2e+6,
    xminorticks        = true,
    yminorticks        = true,
    scaled x ticks     = false,
    clip mode          = individual,
    xlabel             = {frequency $\omega$ (rad/s)},
    ylabel             = {pointwise error $\ekResults$},
    xlabel style       = {yshift = .3em},
    ylabel style       = {yshift = -.3em},
    x tick label style = {/pgf/number format/1000 sep={\,}},
    y tick label style = {/pgf/number format/1000 sep={\,}}
  ]
    \addplot[AAAResponse] table[x = mu, y = AAA]{\tableERR};
    \addplot[LSOAAAResponse] table[x = mu, y = SOLAAA]{\tableERR};
    \addplot[SOAAAResponse] table[x = mu, y = SOAAA]{\tableERR};
    \addplot[NLSOAAAResponse] table[x = mu, y = SONLAAA]{\tableERR};
    \addplot[AAA2kResponse] table[x = mu, y = AAA2k]{\tableERR};
  \end{loglogaxis}
\end{tikzpicture}%
  \tikzexternaldisable%

    \caption{Pointwise weighted error}
    \label{fig:ptwiseErr_porac}
  \end{subfigure}

  \vspace{.5\baselineskip}
  \tikzexternalenable%
  \tikzsetnextfilename{legend_freq}%
  \begin{tikzpicture}
  \begin{axis}[%
    hide axis,
    scale only axis,
    width  = 1cm,
    height = 1cm,
    xmin   = 0,
    xmax   = 1,
    ymin   = 0,
    ymax   = 1,
    legend columns    = -1,
    legend cell align = {left},
    legend style      = {
      at     = {(0,0)},
      anchor = center,
      /tikz/every even column/.append style = {column sep = 0.4cm}}
  ]

    \addplot[trueData] coordinates {(0, 0)};
    \addlegendentry{Original data}

    \addlegendimage{AAAResponse}
    \addlegendentry{\AAA}

    \addlegendimage{LSOAAAResponse}
    \addlegendentry{\AAALin}

    \addlegendimage{SOAAAResponse}
    \addlegendentry{\AAAFull}

    \addlegendimage{NLSOAAAResponse}
    \addlegendentry{\AAAFullNL}

    \addlegendimage{AAA2kResponse}
    \addlegendentry{\AAAtwo}
  \end{axis}
\end{tikzpicture}%
  \tikzexternaldisable%

  \caption{Frequency response and pointwise weighted errors for the acoustic
    cavity example with $\maxOrder = 25$ models.
    The proposed second-order AAA methods can accurately reproduce the
    given data while enforcing the desired system structure.}
    \label{fig:RespAndErrVsFreq_porac}
\end{figure}
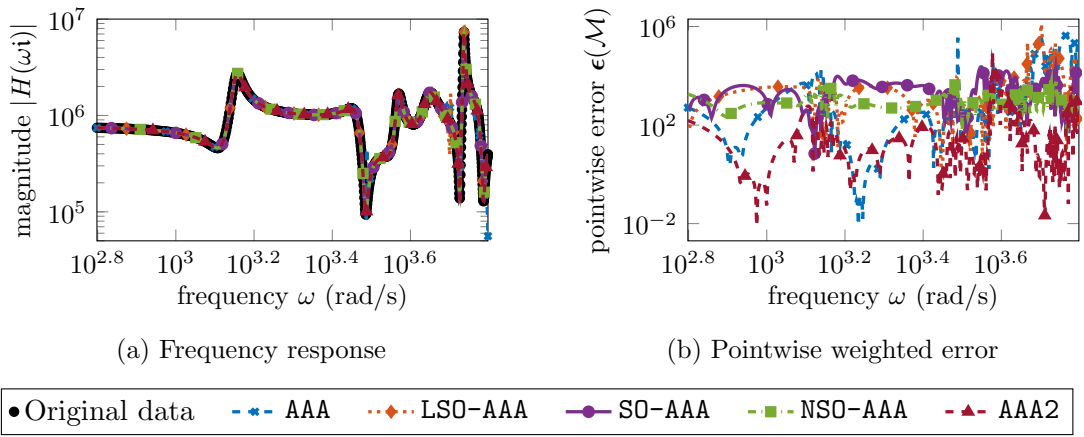


\section{Conclusions}%
\label{sec:conclusions}

In this work, we have developed three new variants of the AAA algorithm for
the data-driven modeling of second-order dynamical systems.
The first variant provides overall good practical performance by balancing
the computational costs and the resulting modeling accuracy.
The two other variants allow to either improve the computation speed of the
method  by sacrificing some modeling accuracy or to improve the modeling
accuracy by investing more computational resources.
Our theoretical analysis suggests that the performance of all three proposed
methods can be bounded by the performance of the classical unstructured AAA
algorithm.
This behavior has been illustrated in the numerical experiments.

A general drawback in AAA-like algorithms is the simplification of the cost
function in the optimization process to provide a computationally efficient
method.
Indeed, our numerical experiments have shown that depending on the numerical
example, there can be large discrepancies between the true, nonlinear
approximation error and the simplified costs that are optimized.
Further investigations into the relation between the simplified cost functions
and the true error behavior will be needed to provide insight into when AAA-like
algorithms can perform well and when not.


\section*{Acknowledgments}%
\addcontentsline{toc}{section}{Acknowledgments}

The work of Gugercin is based upon work supported by the
National Science Foundation under Grant No. AMPS-2318880.


\addcontentsline{toc}{section}{References}
\bibliographystyle{plainurl}
\bibliography{bibtex/myref}

\end{document}